\documentclass{amsart}
\usepackage{amsmath, amsthm, epsfig}
\usepackage{hyperref}
\usepackage{amssymb, latexsym}
\usepackage{url}

\DeclareMathOperator{\Ext}{\mathrm Ext}

\DeclareMathOperator{\rank}{\mathrm rank}
\DeclareMathOperator{\Reg}{\mathrm Reg}

\DeclareMathOperator{\Tor}{\mathrm Tor}

\DeclareMathOperator{\Vol}{\mathrm Vol}

\begin{document}
\bibliographystyle{plain}

 \newtheorem{theorem}{Theorem}[section]
 \newtheorem{lemma}{Lemma}[section]
 \newtheorem{corollary}{Corollary}[section]
 \newtheorem{conjecture}{Conjecture}[section]
 \newcommand{\mc}{\mathcal}
 \newcommand{\A}{\mc A}
 \newcommand{\B}{\mc B}
 \newcommand{\cc}{\mc C}
 \newcommand{\D}{\mc D}
 \newcommand{\E}{\mc E}
 \newcommand{\F}{\mc F }
 \newcommand{\G}{\mc G}
 \newcommand{\hH}{\mc H}
 \newcommand{\I}{\mc I}
 \newcommand{\J}{\mc J}
 \newcommand{\eL}{\mc L}
 \newcommand{\M}{\mc M}
 \newcommand{\eN}{\mc N}
 \newcommand{\pP}{\mc P}
 \newcommand{\qq}{\mc Q}
 \newcommand{\sS}{\mc S}
 \newcommand{\U}{\mc U}
 \newcommand{\V}{\mc V}
 \newcommand{\X}{\mc X}
 \newcommand{\Y}{\mc Y}
 \newcommand{\C}{\mathbb{C}}
 \newcommand{\R}{\mathbb{R}}
 \newcommand{\N}{\mathbb{N}}
 \newcommand{\bP}{\mathbb P}
 \newcommand{\Q}{\mathbb{Q}}
 \newcommand{\T}{\mathbb{T}}
 \newcommand{\Z}{\mathbb{Z}}
 \newcommand{\fA}{\mathfrak A}
 \newcommand{\fB}{\mathfrak B}
 \newcommand{\ff}{\mathfrak F}
 \newcommand{\fI}{\mathfrak I}
 \newcommand{\fJ}{\mathfrak J}
 \newcommand{\fP}{\mathfrak P}
 \newcommand{\fQ}{\mathfrak Q}
 \newcommand{\fU}{\mathfrak U}
 \newcommand{\fb}{f_{\beta}}
 \newcommand{\fg}{f_{\gamma}}
 \newcommand{\gb}{g_{\beta}}
 \newcommand{\vphi}{\varphi}
 \newcommand{\vep}{\varepsilon}
 \newcommand{\bo}{\boldsymbol 0}
 \newcommand{\bone}{\boldsymbol 1}
 \newcommand{\ba}{\boldsymbol a}
 \newcommand{\bb}{\boldsymbol b}
 \newcommand{\bc}{\boldsymbol c}
 \newcommand{\be}{\boldsymbol e}
 \newcommand{\bk}{\boldsymbol k}
 \newcommand{\bm}{\boldsymbol m}
 \newcommand{\bn}{\boldsymbol n}
 \newcommand{\balpha}{\boldsymbol \alpha}
 \newcommand{\Bbeta}{\boldsymbol \beta}
 \newcommand{\bgamma}{\boldsymbol \gamma}
 \newcommand{\bbeta}{\boldsymbol \eta}
 \newcommand{\blambda}{\boldsymbol \lambda}
 \newcommand{\bxi}{\boldsymbol \xi}
 \newcommand{\bt}{\boldsymbol t}
 \newcommand{\bu}{\boldsymbol u}
 \newcommand{\bv}{\boldsymbol v}
 \newcommand{\bx}{\boldsymbol x}
 \newcommand{\bwy}{\boldsymbol y}
 \newcommand{\bw}{\boldsymbol w}
 \newcommand{\bz}{\boldsymbol z}
 \newcommand{\hmu}{\widehat \mu}
 \newcommand{\oK}{\overline{K}}
 \newcommand{\oKt}{\overline{K}^{\times}}
 \newcommand{\oQ}{\overline{\Q}}
 \newcommand{\oq}{\oQ^{\times}}
 \newcommand{\oQt}{\oQ^{\times}/\Tor\bigl(\oQ^{\times}\bigr)}
 \newcommand{\ot}{\Tor\bigl(\oQ^{\times}\bigr)}
 \newcommand{\h}{\frac12}
 \newcommand{\hh}{\tfrac12}
 \newcommand{\dx}{\text{\rm d}x}
 \newcommand{\dbx}{\text{\rm d}\bx}
 \newcommand{\dy}{\text{\rm d}y}
 \newcommand{\dmu}{\text{\rm d}\mu}
 \newcommand{\dnu}{\text{\rm d}\nu}
 \newcommand{\dla}{\text{\rm d}\lambda}
 \newcommand{\dlav}{\text{\rm d}\lambda_v}
 \newcommand{\trho}{\widetilde{\rho}}
 \newcommand{\dtrho}{\text{\rm d}\widetilde{\rho}}
 \newcommand{\drho}{\text{\rm d}\rho}

\title[Heights]{A bound for the exterior\\product of $S$-units}
\author{Shabnam~Akhtari}
\author{Jeffrey~D.~Vaaler}
\subjclass[2010]{11J25, 11R27, 15A75, 05D99}
\keywords{Weil height, exterior product}
\thanks{The authors are grateful to the anonymous referee. Shabnam Akhtari's research was supported by the Simons Foundation Collaboration Grant, Award Number 635880, and by  the National Science Foundation Awards DMS-2001281 and DMS-2327098.}

\address{Department of Mathematics, Pennsylvania State University, University Park, PA 16802 USA}
\email{akhtari@psu.edu}

\address{Department of Mathematics, University of Texas, Austin, Texas 78712 USA}
\email{vaaler@math.utexas.edu}
\numberwithin{equation}{section}

\begin{abstract}  We generalize an inequality for the determinant of a real matrix proved by A.~Schinzel, to 
more general exterior products of vectors in Euclidean space.  We apply this inequality to the logarithmic embedding of
$S$-units contained in a number field $k$.  This leads to a bound for the exterior product of $S$-units expressed as a product of 
heights.  Using a volume formula of P.~McMullen we show that our inequality is sharp up to a constant that depends only on the 
rank of the $S$-unit group but not on the field $k$.  Our inequality is related to a conjecture of F.~Rodriguez Villegas. 
\end{abstract}

\maketitle

\section{Introduction}

Let $k$ be an algebraic number field, $k^{\times}$ its multiplicative group of nonzero elements, and $h : k^{\times} \rightarrow [0, \infty)$ 
the absolute, logarithmic, Weil height (or simply the {\it height}).  In \cite{akhtari2016} we proved inequalities that compare the size of an 
$S$-regulator with the product of heights of a maximal collection of independent $S$-units.  If $k \subseteq l$ are both number fields the
results in \cite{akhtari2022} extend inequalities of this sort to the multiplicative group of relative units.  In the present paper we prove 
analogous inequalities for the exterior product of a collection of independent $S$-units that is not a maximal collection.

At each place $v$ of $k$ we write $k_v$ for the completion of $k$ at $v$.  We use two absolute values $\|\ \|_v$ and $|\ |_v$ from the place $v$.  
The absolute value $\|\ \|_v$ extends the usual archimedean or non-archimedean absolute value on the subfield $\Q$.  Then $|\ |_v$ must be a 
power of $\|\ \|_v$, and we set
\begin{equation}\label{intro1}
|\ |_v = \|\ \|_v^{d_v/d},
\end{equation}
where $d_v = [k_v : \Q_v]$ is the local degree of the extension and $d = [k : \Q]$ is the global degree.  With these normalizations
the height of an algebraic number $\alpha \not= 0$ that belongs to $k$ is given by
\begin{equation}\label{intro3}
h(\alpha) = \sum_v \log^+ |\alpha|_v = \hh \sum_v \bigl|\log |\alpha|_v\bigr|.
\end{equation}
Each sum in (\ref{intro3}) is over the set of all places $v$ of $k$, and the equality between the two sums follows 
from the product formula. 

Let $S$ be a finite set of places of $k$ such that $S$ contains all the archimedean places.  Then
\begin{equation*}\label{first1}
O_S = \big\{\gamma \in k : \|\gamma\|_v \le 1\ \text{for all places}\ v \notin S\big\}
\end{equation*}
is the ring of $S$-integers in $k$, and
\begin{equation*}\label{first3}
O_S^{\times} = \big\{\gamma \in k^{\times} : \|\gamma\|_v = 1\ \text{for all places}\ v \notin S \big\}
\end{equation*}
is the multiplicative group of $S$-units in $O_S$.  The abelian group $O_S^{\times}$ has rank $r$, where $|S| = r+1$, and we assume that $r$ is positive.  
We write $\bx = (x_v)$ for a (column) vector in $\R^{r + 1}$ where the coordinates of $\bx$ are indexed by places $v$ in $S$.  And we write
\begin{equation*}\label{first10}
\|\bx\|_1 = \sum_{v \in S} |x_v|
\end{equation*}
for the $l^1$-norm of $\bx$.  The {\it logarithmic embedding} of $O_S^{\times}$ into $\R^{r+1}$ is the homomorphism defined at each point 
$\alpha$ in $O_S^{\times}$ by
\begin{equation}\label{first6}
\alpha \mapsto \balpha = \bigl(d_v \log \|\alpha\|_v\bigr),
\end{equation}
where the rows of the vector $\balpha$ on the right of (\ref{first6}) are indexed by places $v$ in $S$.  It follows from (\ref{intro1}) and (\ref{intro3}) 
that if $\alpha$ is a point in $O_S^{\times}$ and $\balpha$ is the image of $\alpha$ in $\R^{r+1}$ using the logarithmic embedding (\ref{first6}), then
\begin{equation}\label{first12}
2 [k : \Q] h(\alpha) = \sum_{v \in S} \bigl|d_v \log \|\alpha\|_v\bigr| = \|\balpha\|_1.
\end{equation}

The kernel of the logarithmic embedding (\ref{first6}) is the torsion subgroup
\begin{equation}\label{first18}
\big\{\alpha \in O_S^{\times} : \bigl(d_v \log \|\alpha\|_v\bigr) = \bo\big\} = \Tor\bigl(O_S^{\times}\bigr)
\end{equation}
of all roots of unity in $k^{\times}$.  It is known that (\ref{first18}) is a finite, cyclic group, and from the $S$-unit theorem of 
Dirichlet-Chevalley-Hasse (see \cite[Theorem 3.12]{narkiewicz2010}) we learn that the quotient
\begin{equation*}\label{first23}
\fU_S(k) = O_S^{\times}/\Tor\bigl(O_S^{\times}\bigr)
\end{equation*}
is a free abelian group of rank $r$.   Therefore the logarithmic embedding (\ref{first6}) induces an isomorphism from $\fU_S(k)$ onto 
the discrete subgroup
\begin{equation*}\label{first33}
\Gamma_S(k) = \big\{\bigl(d_v \log \|\alpha\|_v\bigr) : \alpha \in O_S^{\times}\big\} \subseteq \R^{r+1},
\end{equation*}
which is a free group of rank $r$.  It follows from the product formula 
\begin{equation*}\label{first34}
\sum_{v \in S} d_v \log \|\alpha\|_v = 0
\end{equation*}
that $\Gamma_S(k)$ is contained in the $r$-dimensional diagonal subspace
\begin{equation*}\label{first38}
\D_r = \Big\{\bx = (x_v) : \sum_{v \in S} x_v = 0\Big\} \subseteq \R^{r + 1}.
\end{equation*}
The height $h$ is constant on cosets of the quotient group $\fU_S(k)$ and therefore $h$ is well defined as a map
\begin{equation*}\label{first42}
h : \fU_S(k) \rightarrow [0, \infty).
\end{equation*}

Let $\eta_1, \eta_2, \dots , \eta_r$ be multiplicatively independent elements in $\fU_S(k)$ that form a basis for the free group $\fU_S(k)$.  Let
\begin{equation*}\label{first49}
\bbeta_j = \bigl(d_v \log \|\eta_j\|_v\bigr),\quad\text{for $j = 1, 2, \dots , r$,}
\end{equation*}
be the logarithmic embedding of these points in $\Gamma_S(k) \subseteq \D_r$.  Working with the induced $l^1$-norm in the exterior algebra 
$\Ext\bigl(\R^{r+1}\bigr)$ we find that
\begin{equation}\label{first54}
(r + 1) \Reg_S(k) = \|\bbeta_1 \wedge \bbeta_2 \wedge \cdots \wedge \bbeta_r\|_1,
\end{equation}  
where $\Reg_S(k)$ is the $S$-regulator.  More generally, let $\alpha_1, \alpha_2, \dots , \alpha_r$ be multiplicatively independent elements in 
$\fU_S(k)$, and let $\fA \subseteq \fU_S(k)$ be the multiplicative subgroup of rank $r$ which they generate.  Let
\begin{equation*}\label{first56}
\balpha_j = \bigl(d_v \log \|\alpha_j\|_v\bigr),\quad\text{for $j = 1, 2, \dots , r$,}
\end{equation*}
be the image of $\alpha_1, \alpha_2, \dots , \alpha_r$ in $\Gamma_S(k)$.  It follows that there exists a unique $r \times r$ nonsingular matrix
$B = \bigl(b_{i j}\bigr)$ with entries in $\Z$, such that
\begin{equation}\label{first58}
\balpha_j = \sum_{i = 1}^r \bbeta_i b_{i j},\quad\text{for $j = 1, 2, \dots , r$.}
\end{equation}
Then the index of the subgroup $\fA$ in $\fU_S(k)$ is 
\begin{equation}\label{first60}
[\fU_S(k) : \fA] = |\det B|.
\end{equation}
Combining (\ref{first54}), (\ref{first58}), and (\ref{first60}), we find that
\begin{equation}\label{first62}
(r+1) \Reg_S(k) [\fU_S(k) : \fA] = \|\balpha_1 \wedge \balpha_2 \wedge \cdots \wedge \balpha_r\|_1.
\end{equation}
In \cite[Theorem 1.1]{akhtari2016} we proved an upper bound for the $S$-regulator that is equivalent to the identity (\ref{first62}) and the inequality
\begin{equation}\label{first64}
\|\balpha_1 \wedge \balpha_2 \wedge \cdots \wedge \balpha_r\|_1 \le 2^{-r} (r + 1) \prod_{j = 1}^r \|\balpha_j\|_1.
\end{equation}
The following result provides a generalization of (\ref{first64}) to an exterior product of $q$ independent vectors in the free group 
$\Gamma_S(k)$, where $1 \le q \le r$.

\begin{theorem}\label{thmfirst1}  Let $\alpha_1, \alpha_2, \dots , \alpha_q$ be multiplicatively independent points in $\fU_S(k)$, and let
\begin{equation*}\label{first84}
\balpha_j = \bigl(d_v \log \|\alpha_j\|_v\bigr),\quad\text{for $j = 1, 2, \dots , q$,}
\end{equation*}
be the logarithmic embedding of $\alpha_1, \alpha_2, \dots , \alpha_q$ in $\Gamma_S(k)$.  Then we have
\begin{equation}\label{first89}
\|\balpha_1 \wedge \balpha_2 \wedge \cdots \wedge \balpha_q\|_1 \le 2^{-q} C(q, r) \prod_{j = 1}^q \|\balpha_j\|_1,
\end{equation}
where
\begin{equation}\label{first93}
C(q, r) = \min\bigg\{2^q, \biggl(\frac{r+1}{r+1-q}\biggr)^{r+1-q}\bigg\}.
\end{equation}
\end{theorem}

We find that
\begin{equation*}\label{first96}
C(q, r) = 2^q \quad\text{if $2q \le r+1$,}
\end{equation*}
and
\begin{equation*}\label{first101}
C(q, r) = \biggl(\frac{r+1}{r+1-q}\biggr)^{r+1-q}\quad\text{if $r+1 \le 2q$.}
\end{equation*}
In particular we have $C(r, r) = (r + 1)$ so that (\ref{first89}) includes the inequality (\ref{first64}).  By applying (\ref{first12}) it follows that
(\ref{first89}) can be written using the Weil height as
\begin{equation*}\label{first107}
\|\balpha_1 \wedge \balpha_2 \wedge \cdots \wedge \balpha_q\|_1 \le C(q, r) \prod_{j = 1}^q \bigl([k : \Q] h(\alpha_j)\bigr).
\end{equation*}

Let $\alpha_1, \alpha_2, \dots , \alpha_q$ and $\balpha_1, \balpha_2, \dots , \balpha_q$, be as in the statement of Theorem \ref{thmfirst1},
and let $\fA$ be the subgroup of $\Gamma_S(k)$ generated by $\balpha_1, \balpha_2, \dots , \balpha_q$.  Clearly $\fA$ is a free 
group of rank $q$.  It is easy to show that the $l^1$-norm of the exterior product
\begin{equation}\label{first120}
\|\balpha_1 \wedge \balpha_2 \wedge \cdots \wedge \balpha_q\|_1
\end{equation}
depends on the subgroup $\fA$, but does not depend on the choice of generators.  Because of (\ref{first62}) the $l^1$-norm of the exterior 
product (\ref{first120}) extends the $S$-regulator from the group $\Gamma_S(k)$ to subgroups of $\Gamma_S(k)$ having lower rank.  

Alternatively, if $\alpha \not= 1$ belongs to $O_S^{\times}$ and $\balpha \not= \bo$ is the image of $\alpha$ with respect to the logarithmic 
embedding (\ref{first6}), then $\balpha$ and $-\balpha$ are the unique pair of generators of a subgroup of rank $1$ in $\Gamma_S(k)$.  
In view of (\ref{first12}) we may regard $\|\balpha\|_1$ as the height of this subgroup.  Then (\ref{first120}) extends the height to more general
subgroups $\fA \subseteq \Gamma_S(k)$ having rank $q$.  This definition of a height on subgroups is similar to the definition stated
in \cite[equation (6.14)]{vaaler2014}.

In \cite[Theorem 1.2]{akhtari2016} we showed that if $\fA \subseteq \Gamma_S(k)$ is a subgroup with full rank $r$, then there exist 
$r$ linearly independent points in $\fA$ such that the product of their heights is bounded by a number depending only on $r$ multiplied by
\begin{equation}\label{first127}
\Reg_S(k) [\fU_S(k) : \fA]. 
\end{equation} 
The following result generalizes \cite[Theorem 1.2]{akhtari2016} to arbitrary subgroups $\fA \subseteq \Gamma_S(k)$ having positive 
rank $q$ where $1 \le q \le r$.  In this result the $S$-regulator (\ref{first127}) is replaced by the $l^1$-norm (\ref{first120}) of 
the exterior product of a set of generators for the subgroup $\fA$. 

\begin{theorem}\label{thmfirst2}  Let $\fA \subseteq \Gamma_S(k)$ be a subgroup of positive rank $q$, and let the points
\begin{equation*}\label{first134}
\balpha_j = \bigl(d_v \log \|\alpha_j\|_v\bigr),\quad\text{where $j = 1, 2, \dots , q$,}
\end{equation*}
generate the subgroup $\fA$.   Then there exists a subgroup $\fB \subseteq \fA$ of rank $q$ and a set of generators
\begin{equation*}\label{first138}
\Bbeta_j = \bigl(d_v \log \|\beta_j\|_v\bigr),\quad\text{where $j = 1, 2, \dots , q$,}
\end{equation*}
for $\fB$, such that
\begin{equation}\label{first147}
\|\Bbeta_1 \wedge \Bbeta_2 \wedge \dots \wedge \Bbeta_q\|_1 = [\fA : \fB] \|\balpha_1 \wedge \balpha_2 \wedge \cdots \wedge \balpha_q\|_1,
\end{equation}                  
and
\begin{equation}\label{first154}                  
\prod_{j = 1}^q \|\Bbeta_j\|_1 \le (q!) \|\balpha_1 \wedge \balpha_2 \wedge \cdots \wedge \balpha_q\|_1.
\end{equation}
Moreover, we have $[\fA : \fB] \le q!$.
\end{theorem}

By applying (\ref{first12}) we find that the product on the left of (\ref{first154}) can be written using the Weil height as
\begin{equation*}\label{first182}
\prod_{i = 1}^q \|\Bbeta_i\|_1 = 2^q \prod_{j = 1}^q \bigl([k : \Q] h(\beta_j)\bigr).
\end{equation*}
Because the subgroups $\fB \subseteq \fA$ both have rank $q$, the identity (\ref{first147}) follows as in our derivation of (\ref{first60}) from
(\ref{first58}).  

It would be of interest to know if there exist absolute constants $b_0 > 0$ and $b_1 > 1$ such that the factor $q!$ on the right of (\ref{first154}) 
could be replaced by $b_0 b_1^q$.  This could have implications for a conjecture of F. Rodriguez Villegas which we discuss in section 2.           

\section{A conjecture of F. Rodriguez Villegas}

In a well known paper D.~H.~Lehmer \cite{lehmer1933} proposed an important problem about the roots of irreducible polynomials in $\Z[x]$.  
An equivalent form of Lehmer's problem stated using the absolute, logarithmic, Weil height (\ref{intro3}) is this: does there exist an absolute constant
$c > 0$ such that
\begin{equation*}\label{second0}
c \le [\Q(\alpha) : \Q] h(\alpha)
\end{equation*}
whenever $\alpha \not= 0$ is an algebraic number and not a root of unity?  If $\alpha \not= 0$ and $\alpha$ is not a unit, the lower bound
\begin{equation*}\label{second1}
\log 2 \le [\Q(\alpha) : \Q] h(\alpha)
\end{equation*}
follows easily.  Therefore when considering Lehmer's problem we may restrict our attention to algebraic units $\alpha$ which are not roots of unity. Further 
information about Lehmer's problem can be found in \cite[Section 1.6.15]{bombieri2006}, \cite{smyth2008}, and \cite[Section 3.6]{waldschmidt2000}.

Let $S_{\infty}$ be the set of archimedean places of $k$ and assume that $|S_{\infty}| \ge 2$.  We continue to write $|S_{\infty}| = r + 1$ so that the 
logarithmic embedding (\ref{first6}) is an isomorphism from the free group
\begin{equation*}\label{second2}
\fU_{S_{\infty}}(k) = O_{S_{\infty}}/\Tor\bigl(O_{S_{\infty}}^{\times}\bigr)
\end{equation*}
onto the discrete subgroup $\Gamma_{S_{\infty}}(k)$ of rank $r$ contained in the diagonal subspace $\D_r \subseteq \R^{r+1}$.  Then Lehmer's problem 
asks if there exists an absolute constant $c > 0$ such that the inequality
\begin{equation}\label{second3}
c \le 2 [k : \Q] h(\alpha) = \|\balpha\|_1
\end{equation}
holds at all points $\balpha \not = \bo$ in $\Gamma_{S_{\infty}}(k)$.  A generalization of this conjecture to independent subsets 
$\balpha_1, \balpha_2, \dots , \balpha_q$ in $\Gamma_{S_{\infty}}(k)$ with $2 \le q \le r$ was proposed by Bertrand \cite{bertrand1997}.  More
precisely, Bertrand asked if for each integer $2 \le q$ there exists a constant $c_q > 0$ such that
\begin{equation}\label{second4}
c_q \le \|\balpha_1 \wedge \balpha_2 \wedge \cdots \wedge \balpha_q\|_2,
\end{equation}
where the $l^2$-norm of the wedge  product on the right of (\ref{second4}) is the covolume of the subgroup of $\Gamma_{S_{\infty}}(k)$ generated 
by $\balpha_1, \balpha_2, \cdots , \balpha_q$.  Examples found by Siegel \cite{siegel1969} show that the inequality (\ref{second4}) 
cannot hold for $q = 1$.  However, a positive answer for $q \ge 3$ was established by Amoroso and David in \cite{amoroso1999}.    

An alternative generalization of Lehmer's problem to subgroups of rank $q$ has been proposed in a conjecture of F.~Rodriguez Villegas stated in
\cite[Appendix]{chinburg2022}, and also discussed in \cite{amoroso2021} and \cite{chinburg2022}. We state a special case of this conjecture for pure 
wedges.

\begin{conjecture}\label{consecond1} {\sc [F.~Rodriguez Villegas]}  There exist two absolute constants $c_0 > 0$ and $c_1 > 1$ with the 
following property.  If $q$ is an integer such that
\begin{equation*}\label{second8}
1 \le q \le r = \rank \Gamma_{S_{\infty}}(k),
\end{equation*}
and if $\balpha_1, \balpha_2, \dots , \balpha_q$ are linearly independent points in $\Gamma_{S_{\infty}}(k)$, then
\begin{equation}\label{second13}
c_0 c_1^q \le \|\balpha_1 \wedge \balpha_2 \wedge \cdots \wedge \balpha_q\|_1.
\end{equation}
\end{conjecture}

If $q = 1$ then the truth of (\ref{second13}) would solve the problem originally proposed by Lehmer, and if $q = r$ then (\ref{second13}) follows
from a known lower bound for the regulator proved by R.~Zimmert \cite{zimmert1981}.  Thus the conjecture of Rodriguez Villegas interpolates between 
the unsolved problem of Lehmer and Zimmert's result.  It follows from earlier work of Pohst \cite{pohst1978} and Schinzel \cite{schinzel1973} that
Conjecture \ref{consecond1} holds for the collection of totally real algebraic number fields $k$.  

Let $\balpha_1, \balpha_2, \dots , \balpha_q$ be linearly independent points in $\Gamma_{S_{\infty}}(k)$ and let
$\fA \subseteq \Gamma_{S_{\infty}}(k)$ be the subgroup of rank $q$ that they generate.  We have already observed in connection with
(\ref{first120}) that the $l^1$-norm    
\begin{equation*}\label{second29}
\|\balpha_1 \wedge \balpha_2 \wedge \cdots \wedge \balpha_q\|_1
\end{equation*}
depends on the subgroup $\fA$, but does {\it not} depend on the choice of generators.  Thus Conjecture 2.1 can be regarded as a
generalization of Lehmer's problem (reformulated as a conjecture) from subgroups of rank $1$ to more general subgroups of rank $q$ 
where $1 \le q \le r$.    

Here is a related conjecture.

\begin{conjecture}\label{consecond2}  There exist two absolute constants $d_0 > 0$ and $d_1 > 1$ with the 
following property.  If $q$ is an integer such that
\begin{equation*}\label{second18}
1 \le q \le r = \rank \Gamma_{S_{\infty}}(k),
\end{equation*}
and if $\balpha_1, \balpha_2, \dots , \balpha_q$ are linearly independent points in $\Gamma_{S_{\infty}}(k)$, then
\begin{equation*}\label{second23}
d_0 d_1^q \le \|\balpha_1\|_1 \|\balpha_2\|_1 \cdots \|\balpha_q\|_1.
\end{equation*}
\end{conjecture}

It follows from (\ref{first93}) that the constant on the right of (\ref{first89}) is
\begin{equation*}\label{second28}
2^{-q} C(q, r) \le 1.
\end{equation*}
Therefore if the conjectured inequality (\ref{second13}) is correct, then from Theorem \ref{thmfirst1} we also get
\begin{equation*}\label{second33}
c_0 c_1^q \le \|\balpha_1 \wedge \balpha_2 \wedge \cdots \wedge \balpha_q\|_1 \le \prod_{j = 1}^q \|\balpha_j\|_1.
\end{equation*}
Thus Conjecture \ref{consecond1} implies Conjecture \ref{consecond2} with $d_0 = c_0$ and $d_1 = c_1$.

Now assume that Conjecture \ref{consecond2} is correct.  Let $\balpha_1, \balpha_2, \dots , \balpha_q$ be linearly independent points in 
the logarithmic embedding $\Gamma_{S_{\infty}}(k)$, and let $\fA$ be the subgroup of rank $q$ that they generate.
By Theorem \ref{thmfirst2} there exist linearly independent points $\Bbeta_1, \Bbeta_2, \dots , \Bbeta_q$ in $\fA$ such that
\begin{equation}\label{second37}
d_0 d_1^q \le \|\Bbeta_1\|_1 \|\Bbeta_2\|_1 \cdots \|\Bbeta_q\|_1 \le (q!) \|\balpha_1 \wedge \balpha_2 \wedge \cdots \wedge \balpha_q\|_1,
\end{equation}
where the inequality on the left of (\ref{second37}) follows from Conjecture \ref{consecond2}, and the inequality on the right of (\ref{second37}) 
follows from (\ref{first154}).  However, as $q!$ grows faster than an exponential function of $q$, at present we are unable to conclude that 
Conjecture \ref{consecond2} implies Conjecture \ref{consecond1}.  This could change if the factor $q!$ in the inequality (\ref{first154}) could be 
replaced by a factor of the form $b_0 b_1^q$, where $b_0 > 0$ and $b_1 > 1$ are absolute constants.        

\section{Generalization of Schinzel's inequality, I}

For a real number $x$ we write
\begin{equation*}\label{norm10}
x^+ = \max\{0, x\},\quad\text{and}\quad x^- = \max\{0, -x\},
\end{equation*}
so that $x = x^+ - x^-$ and $|x| = x^+ + x^-$.  Let $\bx = (x_n)$ be a (column) vector in $\R^N$.  As in \cite[equation (4.3)]{akhtari2016}, 
the {\it Schinzel norm} is the function 
\begin{equation*}\label{norm15}
\delta : \R^N \rightarrow [0, \infty) 
\end{equation*}
defined by
\begin{equation*}\label{norm20}
\delta(\bx) = \max\bigg\{\sum_{m=1}^N x_m^+, ~\sum_{n=1}^N x_n^-\bigg\} = \hh \biggl|\sum_{n = 1}^N x_n \biggr| + \hh \sum_{n = 1}^N |x_n|.
\end{equation*}
It is clear that $\delta$ is in fact a norm on $\R^N$, and we write
\begin{equation*}\label{norm25}
K_N = \big\{\bx \in \R^N : \delta(\bx) \le 1\big\}
\end{equation*}
for the corresponding closed unit ball.  Then $K_N$ is a compact, convex, symmetric subset of $\R^N$ with a nonempty interior.
The $N$-dimensional volume of $K_N$ was computed in \cite[Lemma 4.1]{akhtari2016}.  The connection between the Schinzel
norm and the Weil height follows from (\ref{first12}) and (\ref{wedge11}) (see also \cite[Lemma 5.1]{akhtari2016}). 

In Lemma \ref{lemreg1} we will determine the finite collection of extreme points of $K_N$.  Then a combinatorial argument in section 4 applied 
to the extreme points of $K_N$ will lead to a proof of the following inequalities.

\begin{theorem}\label{thmnorm1}  Let $\bx_1, \bx_2, \dots, \bx_L$ be linearly independent vectors in $\R^N$.  If $L = N$ then
\begin{equation}\label{norm35}
|\bx_1 \wedge \bx_2 \wedge \cdots \wedge \bx_N| \le \delta(\bx_1) \delta(\bx_2) \cdots \delta(\bx_N),
\end{equation}
if $L < N \le 2L$ then
\begin{equation}\label{norm38}
\|\bx_1 \wedge \bx_2 \wedge \cdots \wedge \bx_L\|_1 \le \biggl(\frac{N}{N-L}\biggr)^{N - L} \delta(\bx_1) \delta(\bx_2) \cdots \delta(\bx_L),
\end{equation}
and if $2L \le N$ then
\begin{equation}\label{norm41}
\|\bx_1 \wedge \bx_2 \wedge \cdots \wedge \bx_L\|_1 \le 2^L\ \delta(\bx_1) \delta(\bx_2) \cdots \delta(\bx_L).
\end{equation}
Alternatively, for $L < N$ we have
\begin{equation}\label{norm44}
\begin{split}
\|\bx_1 \wedge \bx_2 \wedge &\cdots \wedge \bx_L\|_1\\
				&\le \min\bigg\{2^L, \biggl(\frac{N}{N-L}\biggr)^{N - L}\bigg\}\ \delta(\bx_1) \delta(\bx_2) \cdots \delta(\bx_L).
\end{split}
\end{equation}
\end{theorem}

If $\bx_1, \bx_2, \dots , \bx_N$, are (column) vectors in $\R^N$, then Schinzel \cite{schinzel1978} proved the inequality
\begin{equation}\label{norm47}
\bigl|\det\bigl(\bx_1\ \bx_2\ \cdots\ \bx_N\bigr)\bigr| \le \delta\bigl(\bx_1\bigr) \delta\bigl(\bx_2\bigr) \cdots \delta\bigl(\bx_N\bigr),
\end{equation}
which is equivalent to (\ref{norm35}).   It can be shown that there exist nontrivial cases of equality in the inequality (\ref{norm38}) whenever
the integer $N - L$ is a divisor of $N$.  And it can be shown that there always exist nontrivial cases of equality in the inequality (\ref{norm41}). 
It is instructive to define the function
\begin{equation*}\label{reg20}
g_L : [L, \infty] \rightarrow \bigl[1, e^L\bigr]
\end{equation*}
by
\begin{equation*}\label{reg25}
g_L(x) = \begin{cases}   1&     \text{if $x = L$,}\\
                          \biggl(\dfrac{x}{x - L}\biggr)^{(x - L)}&     \text{if $L < x < \infty$},\\
                                      e^L&      \text{if $x = \infty$}. \end{cases}
\end{equation*}
It follows that $x \mapsto g_L(x)$ is continuous, and has a continuous, positive derivative on $(L, \infty)$.  Then
$x \mapsto g_L(x)$ is strictly increasing on $[L, \infty]$.  We have $g_L(2L) = 2^L$, and this clarifies the behavior of the function
\begin{equation*}\label{reg27}
x \mapsto \min\big\{2^L, g_L(x)\big\}
\end{equation*}
which occurs on the right of (\ref{norm44}).  

We recall that a point $\bk$ in $K_N$ is an {\it extreme point} of $K_N$ if $\bk$ cannot be written as a proper convex combination of
two distinct points in $K_N$.  Obviously all extreme points of $K_N$ occur on the boundary of $K_N$.  Let
\begin{equation*}\label{reg31}
\vphi : \R^N \rightarrow \R
\end{equation*}
be a continuous linear functional, and write
\begin{equation*}\label{reg34}
\delta^*(\vphi) = \sup\{\vphi(\bx) : \delta(\bx) \le 1\}
\end{equation*}
for the dual norm of $\vphi$.  As $K_N$ is compact there exists a point $\bbeta$ in $K_N$ such that
\begin{equation*}\label{reg36}
\delta^*(\vphi) = \vphi(\bbeta).
\end{equation*}
If there exists a linear functional $\vphi$ such that
\begin{equation*}\label{reg38}
\big\{\bbeta \in K_N : \delta^*(\vphi) = \vphi(\bbeta)\big\} = \{\bk\},
\end{equation*}
then $\bk$ is an {\it exposed point} of $K_N$.  It is known (see \cite[section 1.8, exercise 3]{eggleston1963}) that an exposed point of 
$K_N$ is also an extreme point of $K_N$.

We define two finite, disjoint subsets of $\R^N$ by
\begin{equation}\label{reg40}
E_N = \big\{\pm \be_m : 1 \le m \le N\big\},\quad\text{and}\quad F_N = \big\{\be_m - \be_n : m \not= n\big\},
\end{equation}
where $\be_1, \be_2, \dots , \be_N$, are the standard basis vectors in $\R^N$.  Clearly we have
\begin{equation*}\label{reg42}
\bigl|E_N\bigr| = 2N,\quad\text{and}\quad \bigl|F_N\bigr| = N^2 - N.
\end{equation*}
It follows easily that each point of $E_N \cup F_N$ is on the boundary of $K_N$. 

\begin{lemma}\label{lemreg1}  The subset $E_N \cup F_N$ is the collection of all extreme points of $K_N$.
\end{lemma}

\begin{proof}  For $1 \le m \le N$ let $\vphi_m : \R^N \rightarrow \R$ be the linear functional defined by
\begin{equation*}\label{reg45}
\vphi_m(\bx) = \hh \sum_{n = 1}^N x_n  +  \hh x_m.
\end{equation*}
Then we have
\begin{equation}\label{reg46}
\vphi_m(\bx) \le \hh \biggl|\sum_{n = 1}^N x_n\biggr|  + \hh \bigl|x_m\bigr|,
\end{equation}
and there is equality in the inequality (\ref{reg46}) if and only if 
\begin{equation*}\label{reg47}
0 \le \sum_{n = 1}^N x_n,\quad\text{and}\quad 0 \le x_m.
\end{equation*}
We also have
\begin{equation}\label{reg48}
\hh \biggl|\sum_{n = 1}^N x_n \biggr| + \hh \bigl| x_m\bigr| \le \delta(\bx),
\end{equation}
and there is equality in the inequality (\ref{reg48}) if and only if
\begin{equation*}\label{reg50}
x_n = 0\quad\text{for each $n \not= m$.}
\end{equation*}
Combining (\ref{reg46}) and (\ref{reg48}) we find that
\begin{equation}\label{reg52}
\vphi_m(\bx) \le \delta(\bx)
\end{equation}
for all $\bx$ in $\R^N$, and there is equality in the inequality (\ref{reg52}) if and only if $\bx = t\be_m$ with $0 \le t$.
Therefore we get
\begin{equation*}\label{reg54}
\delta^*\bigl(\vphi_m\bigr) = \sup\big\{\vphi_m(\bx) : \delta(\bx) \le 1\big\} = \vphi_m\bigl(\be_m\bigr) = 1,
\end{equation*}
and
\begin{equation*}
\big\{\bbeta \in K_N : \delta^*\bigl(\vphi_m\bigr) = \vphi_m(\bbeta)\big\} = \big\{\be_m\big\}.
\end{equation*}
This shows that $\be_m$ is an exposed point of $K_N$, and therefore $\be_m$ is an extreme point of $K_N$.  As $K_N$ is symmetric, 
we find that $-\be_m$ is also an extreme point.

Next we suppose that $m \not= n$, and we define the linear functional $\psi_{mn} : \R^N \rightarrow \R$ by
\begin{equation*}\label{reg56}
\psi_{mn}(\bx) = \hh\bigl(x_m - x_n\bigr).
\end{equation*}
Then we have
\begin{equation}\label{reg58}
\psi_{mn}(\bx) \le \hh\biggl|\sum_{\ell = 1}^N x_{\ell}\biggr| + \hh \bigl|x_m\bigr| + \hh \bigl|x_n\bigr|,
\end{equation}
and there is equality in the inequality (\ref{reg58}) if and only if 
\begin{equation*}\label{reg60}
\sum_{\ell = 1}^N x_{\ell} = 0,\quad 0 \le x_m,\quad\text{and}\quad x_n \le 0.
\end{equation*}
And we get
\begin{equation}\label{reg62}
\hh\biggl|\sum_{\ell = 1}^N x_{\ell}\biggr| + \hh \bigl|x_m\bigr| + \hh \bigl|x_n\bigr| \le \delta(\bx),
\end{equation}
with equality in the inequality (\ref{reg62}) if and only if 
\begin{equation*}\label{reg64}
x_{\ell} = 0\quad\text{for all $\ell \not= m$ and $\ell \not= n$.}
\end{equation*}
By combining (\ref{reg58}) and (\ref{reg62}) we find that
\begin{equation}\label{reg66}
\psi_{mn}(\bx) \le \delta(\bx),
\end{equation}
and there is equality in the inequality (\ref{reg66}) if and only if $\bx = t\bigl(\be_m - \be_n\bigr)$ with $0 \le t$.  As
in the previous case we conclude that
\begin{equation*}\label{reg68}
\delta^*\bigl(\psi_{mn}\bigr) = \sup\big\{\psi_{mn}(\bx) : \delta(\bx) \le 1\big\} = \psi_{mn}\bigl(\be_m - \be_n\bigr) = 1,
\end{equation*}
and
\begin{equation*}\label{reg70}
\big\{\bbeta \in K : \delta^*\bigl(\psi_{mn}\bigr) = \psi_{mn}(\bbeta)\big\} = \big\{\be_m - \be_n\big\}.
\end{equation*}
This shows that $\be_m - \be_n$ is an exposed point of $K_N$, and therefore $\be_m - \be_n$ is an extreme point of $K_N$.

We have now shown that each point in $E_N \cup F_N$ is an extreme point of $K_N$.
To complete the proof we will show that if $\bx$ is a point on the boundary of $K_N$, then $\bx$ can be written as a convex 
combination of points in $E_N \cup F_N$.  Thus we assume that
\begin{equation}\label{reg72}
\delta(\bx) = \max\bigg\{\sum_{m=1}^N x_m^+, ~\sum_{n=1}^N x_n^-\bigg\} = 1,
\end{equation}
and we write
\begin{equation*}\label{reg74}
\sigma^+ = \sum_{m = 1}^N x_m^+,\quad\text{and}\quad \sigma^- = \sum_{n = 1}^N x_n^-.
\end{equation*}
Then we have
\begin{align*}\label{reg76}
\begin{split}
&\underset{m \not= n}{\sum_{m = 1}^N \sum_{n = 1}^N} x_m^+ x_n^- (\be_m - \be_n)\\
 	&= \biggl(\sum_{n = 1}^N x_n^-\biggr) \sum_{m = 1}^N x_m^+ \be_m
		- \biggl(\sum_{m = 1}^N x_m^+\biggr) \sum_{n = 1}^N x_n^- \be_n\\
	&= \sigma^- \sum_{m = 1}^N x_m^+ \be_m - \sigma^+ \sum_{n = 1}^N x_n^- \be_n\\
	&= \sum_{m = 1}^N x_m^+ \be_m - \sum_{n = 1}^N x_n^- \be_n
		- (1 - \sigma^-) \sum_{m = 1}^N x_m^+ \be_m + (1 - \sigma^+) \sum_{n = 1}^N x_n^- \be_n\\
	&= \bx - (1 - \sigma^-) \sum_{m = 1}^N x_m^+ \be_m - (1 - \sigma^+) \sum_{n = 1}^N x_n^- (-\be_n),
\end{split}
\end{align*}
and therefore
\begin{equation}\label{reg78}
\begin{split}
\bx = (1 - \sigma^-) &\sum_{m = 1}^N x_m^+ \be_m + (1 - \sigma^+) \sum_{n = 1}^N x_n^- (-\be_n)\cr
                                 &+ \underset{m \not= n}{\sum_{m = 1}^N \sum_{n = 1}^N} x_m^+ x_n^- (\be_m - \be_n).
\end{split}
\end{equation}
The identity (\ref{reg78}) shows that $\bx$ is a linear combination of points in $E_N \cup F_N$ with nonnegative
coefficients.  Using (\ref{reg72}), the sum of the coefficients in (\ref{reg78}) is
\begin{equation*}\label{reg80}
\begin{split}
(1 - \sigma^-) \sum_{m = 1}^N x_m^+ + (1 - \sigma^+) &\sum_{n = 1}^N x_n^- 
						+ \underset{m \not= n}{\sum_{m = 1}^N \sum_{n = 1}^N} x_m^+ x_n^-\\
		&= (1 - \sigma^-) \sigma^+ + (1 - \sigma^+) \sigma^- + \sigma^+ \sigma^-\\
		&= 1 - (1 - \sigma^+)(1 - \sigma^-)\\
		&= 1.
\end{split}
\end{equation*}
It follows that $\bx$ is a convex combination of points in $E_N \cup F_N$.
We have shown that if $\bx$ is on the boundary of $K_N$, then $\bx$ is a convex combination of points in $E_N \cup F_N$.
Therefore the only extreme points of $K_N$ are the points in $E_N \cup F_N$.
\end{proof}

Let 
\begin{equation*}\label{reg82}
I = \{i_1 < i_2 < \cdots < i_L\} \subseteq \{1, 2, \dots , N\}
\end{equation*}
be a subset of positive cardinality $L$.  If $\bx = (x_n)$ is a point in $\R^N$ we write $\bx_I$ for the point in $\R^L$ given by
$\bx_I = (x_{i_{\ell}})$.  Alternatively, $\bx_I$ is the $L\times 1$ submatrix of $\bx$ having rows indexed by the integers in the subset $I$.
The following result is now an immediate consequence of Lemma \ref{lemreg1}.

\begin{corollary}\label{correg1}  Let $\bxi$ be an element in the set of extreme points $E_N \cup F_N$, and let 
\begin{equation*}\label{reg83}
I \subseteq \{1, 2, \dots , N\}
\end{equation*}
be a subset of positive cardinality $L$.  Then either $\bxi_I = \bo$ in $\Z^L$, or $\bxi_I$ belongs to the set of extreme points $E_L\cup F_L$.
\end{corollary}

Let 
\begin{equation*}\label{reg85}
\Phi_{L, N} : \R^N \times \R^N \times \cdots \times \R^N \rightarrow \R^M,\quad\text{where $M = \binom{N}{L}$},
\end{equation*}
be the continuous, alternating, multilinear function taking values in $\R^M$ and defined by
\begin{equation*}\label{reg87}
\Phi_{L, N}(\bx_1, \bx_2, \dots , \bx_L) = \bx_1 \wedge \bx_2 \wedge \cdots \wedge \bx_L.
\end{equation*}
By compactness the continuous, nonnegative function
\begin{equation*}\label{reg91}
(\bx_1, \bx_2, \dots , \bx_L) \mapsto \|\bx_1 \wedge \bx_2 \wedge \cdots \wedge \bx_L\|_1
\end{equation*}
assumes its maximum value on the $L$-fold product
\begin{equation*}\label{reg93}
K_N \times K_N \times \cdots \times K_N.
\end{equation*}
We write
\begin{equation}\label{reg95}
\mu_{L, N} = \max\big\{\|\bx_1 \wedge \bx_2 \wedge \cdots \wedge \bx_L\|_1 : \text{$\bx_{\ell} \in K_N$ for $\ell = 1, 2, \dots , L$}\big\}
\end{equation}
for this maximum value. We show that $\mu_{L, N}$ can be determined by restricting each variable $\bx_{\ell}$ to the set $E_N \cup F_N$ of 
extreme points in $K_N$.

\begin{lemma}\label{lemreg2}  There exist points $\bxi_1, \bxi_2, \dots , \bxi_L$ in the set of extreme points $E_N \cup F_N$ such that
\begin{equation}\label{reg97}
\mu_{L, N} = \|\bxi_1 \wedge \bxi_2 \wedge \cdots \wedge \bxi_L\|_1.
\end{equation}
Moreover, if $\bx_1, \bx_2, \dots , \bx_L$ are vectors in $\R^N$ then 
\begin{equation}\label{reg99}
\|\bx_1 \wedge \bx_2 \wedge \cdots \wedge \bx_L\|_1 \le \mu_{L, N}\ \delta(\bx_1) \delta(\bx_2) \cdots \delta(\bx_L).
\end{equation}
\end{lemma}

\begin{proof}  Let $\bbeta_1, \bbeta_2, \dots , \bbeta_L$, be points in $K_N$ such that 
\begin{equation}\label{reg100}
\mu_{L, N} = \bigl\|\bbeta_1 \wedge \bbeta_2 \wedge \cdots \wedge \bbeta_L\bigr\|_1.
\end{equation}
Because $\Phi_{L, N}$ is linear in each variable, it is easy to show that $\delta\bigl(\bbeta_{\ell}\bigr) = 1$ for each
$\ell = 1, 2, \dots , L$.  Also, among all the collections of $L$ points from the boundary of $K_N$ that satisfy (\ref{reg100}), we
may assume that the collection $\bbeta_1, \bbeta_2, \dots , \bbeta_L$ contains the maximum 
number of extreme points.  If this maximum number is $L$ then we are done.  Therefore we may assume that the maximum number 
of extreme points is less than $L$.

If, for example, $\bbeta_1$ is not an extreme point, then there exist extreme points $\bu_1, \bu_2, \dots , \bu_J$
in $K_N$, and positive numbers $\theta_1, \theta_2, \dots , \theta_J$, such that
\begin{equation*}\label{reg102}
\bbeta_1 = \sum_{j = 1}^J \theta_j \bu_j,\quad\text{and}\quad \sum_{j = 1}^J \theta_j = 1.
\end{equation*}
It follows that
\begin{equation}\label{reg104}
\begin{split}
\mu_{L, N} &= \biggl\|~\sum_{j = 1}^J\theta_j \bigl(\bu_j \wedge \bbeta_2 \wedge \cdots \wedge \bbeta_L\bigr)\biggr\|_1\\
                  &\le \sum_{j = 1}^J \theta_j \bigl\|\bu_j \wedge \bbeta_2 \wedge \cdots \wedge \bbeta_L\bigr\|_1\\
                  &\le \mu_{L, N} \sum_{j = 1}^J\theta_j\\
                  &= \mu_{L, N}
\end{split}
\end{equation}
Hence there is equality throughout the inequality (\ref{reg104}), and we conclude that
\begin{equation*}\label{reg106}
\mu_{L, N} = \bigl\|\bu_j \wedge \bbeta_2 \wedge \cdots \wedge \bbeta_L\bigr\|_1
\end{equation*}
for each $j = 1, 2, \dots , J$.  But each collection of points $\bu_j, \bbeta_2, \dots , \bbeta_L,$ plainly contains
one more extreme point than the collection $\bbeta_1, \bbeta_2, \dots , \bbeta_L$.  The contradiction shows
that there exists a collection of points $\bxi_1, \bxi_2, \dots , \bxi_L$, from the boundary of $K_N$, such that (\ref{reg97}) holds
and each $\bxi_{\ell}$ is an extreme point of $K_N$.

Next we verify the inequality (\ref{reg99}).  If one of the vectors in the collection $\bx_1, \bx_2, \dots , \bx_L$ is the zero vector, then both 
sides of (\ref{reg99}) are zero.  Thus we may assume that $\bx_{\ell} \not= \bo$ for each $\ell = 1, 2, \dots , L$.  Let 
\begin{equation}\label{reg112}
\bwy_{\ell} = \delta\bigl(\bx_{\ell}\bigr)^{-1} \bx_{\ell},
\end{equation}
so that $\delta\bigl(\bwy_{\ell}\bigr) = 1$ for each $\ell = 1, 2, \dots , L$.  Then we certainly have
\begin{equation}\label{reg114}
\bigl\|\bwy_1 \wedge \bwy_2 \wedge \cdots \wedge \bwy_L\bigr\|_1 \le \mu_{L, N}
\end{equation}
by the definition of $\mu_{L, N}$.  Then (\ref{reg99}) follows using (\ref{reg112}), (\ref{reg114}), and the multilinearity of the
exterior product.
\end{proof}

The extreme points $E_N \cup F_N$ for the $\delta$-unit ball $K_N$ have the following useful property.

\begin{lemma}\label{lemreg3}  Let $\bxi_1, \bxi_2, \dots , \bxi_L$ be extreme points in the set $E_N\cup F_N$, and let
\begin{equation*}\label{turn10}
\Xi = \bigl(\bxi_1\ \bxi_2\ \cdots \ \bxi_L\bigr)
\end{equation*}
be the $N\times L$ matrix having $\bxi_1, \bxi_2, \dots , \bxi_L$ as columns.  If
\begin{equation*}\label{turn15}
I \subseteq \{1, 2, \dots , N\}
\end{equation*}
is a subset of cardinality $|I| = L$, and $\Xi_I$ is the $L \times L$ submatrix having rows indexed by $I$, then the integer
$\det \Xi_I$ belongs to the set $\{-1, 0, 1\}$.
\end{lemma}

\begin{proof}  It is clear that the columns of the $L \times L$ submatrix $\Xi_I$ are the $L \times 1$ column vectors 
$(\bxi_1)_I, (\bxi_2)_I, \dots , (\bxi_L)_I$.  If a column of $\Xi_I$ is $\bo$, then $\det \Xi_I = 0$ is obvious.
If each column of $\Xi_I$ is not $\bo$, then it follows from Corollary \ref{correg1} that each column of $\Xi_I$ belongs to the set 
of extreme points $E_L \cup F_L$.  Applying Schinzel's determinant inequality (\ref{norm47}) to the matrix $\Xi_I$, we get
\begin{equation*}\label{turn20}
\bigl|\det \Xi_I\bigr| \le \delta\bigl((\bxi_1)_I\bigr) \delta\bigl((\bxi_2)_I\bigr) \cdots \delta\bigl((\bxi_L)_I\bigr) = 1.
\end{equation*}
As $\det \Xi_I$ is an integer, the lemma is proved.
\end{proof}

If $\bxi_1, \bxi_2, \dots , \bxi_L$ are extreme points in $E_N \cup F_N$, then it follows from Lemma \ref{lemreg3} that
\begin{equation}\label{turn28}
\big\|\bxi_1 \wedge \bxi_2 \wedge \cdots \wedge \bxi_L\big\|_1 
	= \sum_{\substack{I \subseteq \{1, 2, \dots , N\}\\|I| = L}} \bigl|\det \Xi_I\bigr| \le \binom{N}{L}.
\end{equation}
Using (\ref{reg97}) we get the simple upper bound
\begin{equation}\label{turn31}
\mu_{L, N} \le \binom{N}{L}\quad\text{for $1 \le L \le N$}.
\end{equation}
It follows from (\ref{norm47}) that there is equality in (\ref{turn31}) when $L = N$.  There is also equality in (\ref{turn31}) when $L + 1 = N$;
this follows from the example
\begin{equation*}\label{turn34}
\Xi = \begin{pmatrix}  1& 0& 0& \dots& 0& 0\\
                                   0& 1& 0& \dots & 0& 0\\ 
                                   0& 0& 1& \dots & 0& 0\\
                                   0& 0& 0& \dots & 0& 0\\
                   \vdots& \vdots& \vdots& \cdots & \vdots& \vdots\\
                                   0& 0& 0& \dots & 1& 0\\
                                   0& 0& 0& \dots & 0& 1\\
                                  -1& -1& -1& \dots & -1& -1\\
                                     \end{pmatrix}.
\end{equation*}
By squaring each of the subdeterminants in the sum (\ref{turn28}) we can determine the value of $\mu_{L, N}$ for $2L \le N$.

\begin{lemma}\label{lemreg4}  If $1 \le L < N$ then
\begin{equation}\label{turn40}
\mu_{L, N} \le 2^L.
\end{equation}
Moreover, if $2L \le N$ then there is equality in the inequality {\rm (\ref{turn40})}.
\end{lemma}

\begin{proof}  Let $\bxi_1, \bxi_2, \dots , \bxi_L$ be extreme points in $E_N \cup F_N$, and let
\begin{equation*}\label{turn25}
\Xi = \bigl(\bxi_1\ \bxi_2\ \cdots \ \bxi_L\bigr)
\end{equation*}
be the $N\times L$ matrix having $\bxi_1, \bxi_2, \dots , \bxi_L$ as columns.  It follows from Lemma \ref{lemreg3} that
\begin{equation*}\label{turn33}
\big\|\bxi_1 \wedge \bxi_2 \wedge \cdots \wedge \bxi_L\big\|_1 = \sum_{\substack{I \subseteq \{1, 2, \dots , N\}\\|I| = L}} \bigl|\det \Xi_I\bigr|
		= \sum_{\substack{I \subseteq \{1, 2, \dots , N\}\\ |I| = L}} \bigl(\det \Xi_I\bigr)^2.
\end{equation*}
Then from the Cauchy-Binet identity we get
\begin{equation}\label{turn35}
\big\|\bxi_1 \wedge \bxi_2 \wedge \cdots \wedge \bxi_L\big\|_1 
		= \sum_{\substack{I \subseteq \{1, 2, \dots , N\}\\ |I| = L}} \bigl(\det \Xi_I\bigr)^2 = \det\bigl(\Xi^T \Xi\bigr).
\end{equation}
The $L \times L$ matrix in the determinant on the right of (\ref{turn35}) is
\begin{equation*}\label{turn39}
\Xi^T \Xi = \bigl(\bxi_k^T \bxi_{\ell}\bigr),
\end{equation*}
where $k = 1, 2, \dots , L$ indexes rows and $\ell = 1, 2, \dots , L$ indexes columns.  As $\Xi^T \Xi$ is an $L \times L$ real, symmetric
matrix, we can apply Hadamard's inequality to estimate its determinant.  We find that
\begin{equation}\label{turn50}
\big\|\bxi_1 \wedge \bxi_2 \wedge \cdots \wedge \bxi_L\big\|_1 = \det\bigl(\Xi^T \Xi\bigr) \le \prod_{\ell = 1}^L \|\bxi_{\ell}\|_2^2 \le 2^L.
\end{equation}
This proves the inequality (\ref{turn40}).

If the columns of the matrix $\Xi$ are orthogonal, then there is equality in Hadamard's inequality.  Therefore, if $2L \le N$ we
select $\bxi_1, \bxi_2, \dots , \bxi_L$ in $F_N$ so that
\begin{equation*}\label{turn55}
\Xi = \begin{pmatrix}  1& 0& 0& \dots& 0& 0\\
                       -1& 0& 0& \dots & 0& 0\\ 
                        0& 1& 0& \dots & 0& 0\\
                        0& -1& 0& \dots & 0& 0\\
                        0& 0& 1& \dots & 0& 0\\
                        0& 0& -1& \dots & 0& 0\\
        \vdots& \vdots& \vdots& \cdots & \vdots& \vdots\\
                        0& 0& 0& \dots & 1& 0\\
                        0& 0& 0& \dots & -1& 0\\
                        0& 0& 0& \dots & 0& 1\\
                        0& 0& 0& \dots & 0& -1\\
                        0& 0& 0& \dots & 0& 0\\
         \vdots& \vdots& \vdots& \cdots & \vdots& \vdots\\
                         0& 0& 0& \dots & 0& 0\\
                        \end{pmatrix}.
\end{equation*}
For this choice of $\Xi$ the columns of $\Xi$ are orthogonal.  Hence for this choice of $\Xi$ there is equality in (\ref{turn50}), and equality 
in (\ref{turn40}).
\end{proof}

If $\bx_1, \bx_2, \dots , \bx_L$ belong to $\R^N$ and $2L \le N$, then it follows from (\ref{reg99}) and the case of equality in (\ref{turn40}) that
\begin{equation}\label{reg271}
\|\bx_1 \wedge \bx_2 \wedge \cdots \wedge \bx_L\|_1 \le 2^L\ \delta(\bx_1) \delta(\bx_2) \cdots \delta(\bx_L).
\end{equation}
This proves the inequality (\ref{norm41}) in the statement of Theorem \ref{thmnorm1}.  

The following lemma, together with combinatorial arguments in section 4, will be used in the proof of the inequality (\ref{norm38}).

\begin{lemma}\label{lemreg5}  Let $\bxi_1, \bxi_2, \dots , \bxi_L$ be linearly independent extreme points in the set $E_N\cup F_N$.  Assume 
that exactly $K$ of the points $\bxi_1, \bxi_2, \dots , \bxi_L$ belong to the subset $E_N$, where $1 \le K < L$.  Then there exist linearly 
independent extreme points 
$\bbeta_1, \bbeta_2, \dots , \bbeta_{L-K}$ in the set $E_{N-K} \cup F_{N-K}$ such that
\begin{equation*}\label{reg275}
\|\bxi_1 \wedge \bxi_2 \wedge \cdots \wedge \bxi_L \|_1 = \|\bbeta_1 \wedge \bbeta_2 \wedge \cdots \wedge \bbeta_{L-K}\|_1.
\end{equation*}
\end{lemma}

\begin{proof}  By using a suitable permutation of the points $\bxi_1, \bxi_2, \dots , \bxi_L$, we may assume that 
\begin{equation*}\label{reg285}
\{\bxi_1, \bxi_2, \dots , \bxi_K\} \subseteq E_N,\quad\text{and}\quad \{\bxi_{K+1}, \bxi_{K+2}, \dots , \bxi_L\} \subseteq F_N.
\end{equation*}
And we may further assume that for $k = 1, 2, \dots , K$ we have
\begin{equation*}\label{reg288}
\bxi_k = \pm \be_{m_k},\quad\text{where $1 \le m_1 < m_2 < \cdots < m_K \le N$}.
\end{equation*}
It will be convenient to write
\begin{equation*}\label{reg291}
M = \{m_1, m_2, \dots , m_K\}.
\end{equation*}

Now let
\begin{equation*}\label{reg280}
\Xi = \bigl(\bxi_1\ \bxi_2\ \cdots \ \bxi_L\bigr)
\end{equation*}
be the $N\times L$ matrix having $\bxi_1, \bxi_2, \dots , \bxi_L$ as columns.  We partition $\Xi$ into submatrices
\begin{equation*}\label{turn117}
\Xi = \bigl(U\ V\bigr),
\end{equation*}
where
\begin{equation*}\label{turn121}
U = \bigl(\bxi_1\ \bxi_2\ \cdots\ \bxi_K\bigr),\quad\text{and}\quad V = \bigl(\bxi_{K+1}\ \bxi_{K+2}\ \cdots\ \bxi_L\bigr)
\end{equation*}
are $N\times K$ and $N\times (L-K)$, respectively.  We suppose that $I \subseteq \{1, 2, \dots , N\}$ is a subset of cardinality $|I| = L$ such that
\begin{equation}\label{turn131}
\det \Xi_I = \det \bigl(U_I\ V_I\bigr) \not= 0.
\end{equation} 
On the right of (\ref{turn131}) the submatrix $U_I$ is $L \times K$ and the submatrix $V_I$ is $L \times (L - K)$.
If the integer $m_k$, which occurs in $M$, does not belong to $I$, then the $k$-th column of $\Xi_I$ is identically zero and 
(\ref{turn131}) cannot hold.  Therefore (\ref{turn131}) implies that
\begin{equation*}\label{turn137}
M \subseteq I.
\end{equation*}
Next we apply the Laplace expansion of the determinant to $\Xi_I$ partitioned as in (\ref{turn131}).  In view of our previous remarks we find that
\begin{equation}\label{turn143}
\det \Xi_I = \sum_{\substack{J \subseteq I\\|J| = K}} (-1)^{\vep(J)} \bigl(\det U_J\bigr) \bigl(\det V_{\widetilde{J}}\bigr),
\end{equation}
where 
\begin{equation*}\label{turn148}
\widetilde{J} = I \setminus J
\end{equation*}
is the complement of $J$ in $I$, and $\vep(J)$ is an integer that depends on $J$.  As before, if the integer $m_k$ which occurs in $M$
does not belong to the subset $J$, then the $k$-th column of $U_J$ is identically zero and therefore $\det U_J = 0$.  As $|J| = |M| = K$, 
we conclude that there is exactly one nonzero term in the sum on the right of (\ref{turn143}), and the nonzero term occurs when $J = M$.
From these observations we conclude that the Laplace expansion (\ref{turn143}) is simply
\begin{equation}\label{turn163}
\det \Xi_I =  (-1)^{\vep(M)} \bigl(\det U_M\bigr) \bigl(\det V_{I \setminus M}\bigr).
\end{equation}
It is obvious that $\det U_M = \pm 1$, and therefore (\ref{turn163}) leads to the identity
\begin{equation*}\label{turn170}
\bigl|\det \Xi_I\bigr| = \bigl|\det V_{I \setminus M}\bigr|.
\end{equation*}

Let 
\begin{equation*}\label{turn185}
V^{\prime} = \bigl(\bxi_{K+1}^{\prime}\ \bxi_{K+2}^{\prime}\ \cdots \ \bxi_L^{\prime}\bigr)
\end{equation*}
be the $(N-K) \times (L-K)$ submatrix of $V$ obtained by removing the rows of $V$ that are indexed by the integers $m_k$ in the subset $M$.
It follows from Lemma \ref{lemreg2} that the columns of $V^{\prime}$ belong to the set of extreme points $E_{N-K} \cup F_{N-K}$.  Moreover, we have
\begin{equation}\label{turn190}
\bigl|\det \Xi_I\bigr| = \bigl|\det V_{I \setminus M}\bigr| = \bigl|\det V_J^{\prime}\bigr|,
\end{equation}
where
\begin{equation*}\label{turn195}
J = I \setminus M \subseteq \{1, 2, \dots , N\} \setminus M,\quad\text{and}\quad |J| = L - K.
\end{equation*}
We note that 
\begin{equation*}\label{turn205}
I \mapsto J = I \setminus \{m_1, m_2, \dots , m_K\}
\end{equation*}
is a bijection from the collection of subsets $I$ that contain $M$ onto the collection of subsets of $\{1, 2, \dots , N\} \setminus M$ that have 
cardinality $L - K$.  Using (\ref{turn190}) we find that
\begin{equation}\label{turn210}
\sum_{\substack{I \subseteq \{1, 2, \dots , N\}\\M \subseteq I}} \bigl|\det \Xi_I\bigr| 
			= \sum_{\substack{J \subseteq \{1, 2, \dots , N\} \setminus M\\|J| = L-K}} \bigl|\det V_J^{\prime}\bigr|.
\end{equation}
Because the rows of $V^{\prime}$ are indexed by the elements of the set $\{1, 2, \dots , N\} \setminus M$, it follows from (\ref{turn210}) that
\begin{equation}\label{turn215}
\begin{split}
\|\bxi_1 \wedge \bxi_2 \wedge \cdots \wedge \bxi_L\|_1 &= \sum_{\substack{I \subseteq \{1, 2, \dots , N\}\\M \subseteq I}} \bigl|\det \Xi_I\bigr|\\ 
			                                            &= \sum_{\substack{J \subseteq \{1, 2, \dots , N\} \setminus M\\|J| = L-K}} \bigl|\det V_J^{\prime}\bigr|\\
			                                            &= \|\bxi_{K+1}^{\prime} \wedge \bxi_{K+2}^{\prime} \wedge \cdots \wedge \bxi_L^{\prime}\|_1.
\end{split}
\end{equation}
As the columns of $V^{\prime}$ belong to $E_{N-K} \cup F_{N-K}$ and satisfy (\ref{turn215}), they are linearly independent.  Therefore we set
\begin{equation*}\label{trun220}
\bbeta_{\ell} = \bxi_{K + \ell}^{\prime},\quad\text{for $\ell = 1, 2, \dots , L-K$},
\end{equation*}
and the lemma is proved.
\end{proof}

\section{Generalization of Schinzel's inequality, II}

In this section we develop a combinatorial method which leads to an asymptotically sharp upper bound for the quantity
$\mu_{L, N}$ defined in (\ref{reg95}).  The bound we prove here applies when $L < N \le 2L$, and will be used to verify the inequality
(\ref{norm38}) in the statement of Theorem \ref{thmnorm1}.

We suppose throughout this section that
\begin{equation}\label{group95}
\big\{S(1), S(2), S(3), \dots , S(L)\big\}
\end{equation}
is a collection of $L$ distinct subsets of $\{1, 2, \dots , N\}$ such that
\begin{equation}\label{group100}
|S(\ell)| = 2,\quad\text{for each $\ell = 1, 2, \dots , L$},
\end{equation}
and
\begin{equation}\label{group105}
\bigcup_{\ell = 1}^L S(\ell)= \{1, 2, \dots , N\}.
\end{equation}
It follows from (\ref{group100}) and (\ref{group105}) that
\begin{equation*}\label{group106}
N \le 2L \le N(N - 1),
\end{equation*}
but for our later applications we will make the more restrictive assumption that
\begin{equation}\label{group107}
L < N \le 2L.
\end{equation}

Let $\A$ be the collection of {\it all} subsets $A \subseteq \{1, 2, \dots , N\}$.  We define a map $\eta : \A \rightarrow \A$ by
\begin{equation}\label{group135}
\eta(A) = \bigcup_{\substack{\ell = 1\\ S(\ell) \cap A \not= \emptyset}}^L S(\ell).
\end{equation}
Then it follows from (\ref{group105}) that
\begin{equation}\label{group140}
A \subseteq \eta(A),\quad\text{for each subset $A \in \A$}.
\end{equation}
We are interested in subsets $A$ in $\A$ that satisfy $\eta(A) = A$.  Obviously $\emptyset$ and $\{1, 2, \dots , N\}$ have
this property.  More generally we define
\begin{equation}\label{group142}
\pP = \big\{A \in \A : \eta(A) = A\big\}.
\end{equation}
If $A$ belongs to the collection $\pP$ and $S(\ell) \cap A \not= \emptyset$, then $S(\ell) \subseteq A$.  Thus a nonempty subset
$A$ in $\pP$ must have $2 \le |A|$.  We show that the collection $\pP$ forms an algebra of subsets.
 
\begin{lemma}\label{lemgroup1}  Let $\pP \subseteq \A$ be the collection of subsets defined by {\rm (\ref{group142})}.
\begin{itemize}
\item[(i)]  If $A_1$ belongs to $\pP$ then its complement
\begin{equation*}\label{group144}
A_2 = \{1, 2, \dots , N\} \setminus A_1
\end{equation*}
also belongs to $\pP$.
\item[(ii)]  If $A_3$ and $A_4$ belong to $\pP$ then $A_3 \cup A_4$ belongs to $\pP$.
\item[(iii)]  If $A_5$ and $A_6$ belong to $\pP$ then $A_5 \cap A_6$ belongs to $\pP$.
\end{itemize}
\end{lemma}

\begin{proof}  Assume that $S(\ell) \cap A_2 \not= \emptyset$.  Then $S(\ell) \cap A_1 \not= \emptyset$ is impossible.  Hence we
have $S(\ell) \subseteq A_2$, and this implies that $A_2$ belongs to $\pP$.

Let $S(\ell) \cap (A_3 \cup A_4) \not= \emptyset$.  Then either $S(\ell) \cap A_3 \not= \emptyset$ or $S(\ell) \cap A_4 \not= \emptyset$.
Hence either $S(\ell) \subseteq A_3$ or $S(\ell) \subseteq A_4$, and therefore $S(\ell) \subseteq A_3 \cup A_4$.  It follows that
$A_3 \cup A_4$ belongs to $\pP$.

By what we have already proved the sets
\begin{equation*}\label{group146}
A_7 = \{1, 2, \dots , N\} \setminus A_5,\quad\text{and}\quad A_8 = \{1, 2, \dots , N\} \setminus A_6
\end{equation*}
both belong to $\pP$, and therefore the set
\begin{equation*}\label{group148}
A_5 \cap A_6 = \{1, 2, \dots , N\} \setminus (A_7 \cup A_8)
\end{equation*}
belongs to $\pP$.
\end{proof}

\begin{lemma}\label{lemgroup2}  Let $A_1$ be a nonempty subset in $\A$, and let $B$ be a subset in $\pP$.  Assume that $A_1 \subseteq B$.  
Define an increasing sequence of subsets 
\begin{equation*}\label{group150}
A_1, A_2, A_3, \dots 
\end{equation*}
from $\A$ inductively by
\begin{equation*}\label{group152}
A_{n+1} = \eta\bigl(A_n\bigr),\quad\text{for each $n = 1, 2, 3, \dots $}.
\end{equation*}
Then 
\begin{equation*}\label{group153}
A_n \subseteq B\quad\text{for each $n = 1, 2, 3, \dots $}.
\end{equation*}
\end{lemma}

\begin{proof}  We argue by induction on $n$.  If $n = 1$ then $A_1 \subseteq B$ by hypothesis.  Now assume that $2 \le n$
and $A_{n-1} \subseteq B$.  Then we have
\begin{equation}\label{group165}
A_n = \eta(A_{n-1}) = \bigcup_{\substack{\ell = 1\\S(\ell) \cap A_{n-1} \not= \emptyset}}^L S(\ell).
\end{equation}
If $S(\ell) \cap A_{n-1} \not= \emptyset$ then $S(\ell)$ contains a point of $B$, and therefore 
$S(\ell) \subseteq B$.  It follows from (\ref{group165}) that $A_n \subseteq B$.  This proves the lemma.
\end{proof}

We say that a subset $A$ in $\A$ is {\it minimal} if $A$ is not empty and belongs to $\pP$, but no proper subset of $A$ belongs to $\pP$.  That is, 
a nonempty set $A$ in $\pP$ is {\it minimal} if for every nonempty subset $B \subseteq A$ such that $B \not= A$, we have $\eta(B) \not= B$.
We will show that each element of $\{1, 2, \dots , N\}$ is contained in a minimal subset in $\pP$.

\begin{lemma}\label{lemgroup3}  Let $A_1$ in $\A$ have cardinality $1$.  Define an increasing sequence of subsets 
\begin{equation*}\label{group158}
A_1, A_2, A_3, \dots 
\end{equation*}
from $\A$ inductively by
\begin{equation}\label{group160}
A_{n+1} = \eta\bigl(A_n\bigr),\quad\text{for $n = 1, 2, 3, \dots $}.
\end{equation}
Let $K$ be the smallest positive integer such that
\begin{equation}\label{group162}
A_K = \eta(A_K) = A_{K+1}.
\end{equation}
Then $K$ exists, $2 \le K$, and the subset $A_K$ is minimal.
\end{lemma}

\begin{proof}  From (\ref{group140}) we get
\begin{equation*}\label{group164}
A_1 \subseteq A_2 \subseteq A_3 \subseteq \cdots \subseteq A_n \subseteq \cdots.
\end{equation*}
As $|A_n| \le N$ for each $n = 1, 2, \dots $, it is obvious that $K$ exists.

Let $A_1 = \{k_1\}$ where $1 \le k_1 \le N$.  It follows from (\ref{group105}) that there exists a subset $S(\ell_1)$ that contains $k_1$.
Write $S(\ell_1) = \{k_1, k_2\}$ where $k_1 \not= k_2$.  From (\ref{group135}) we conclude that 
\begin{equation*}\label{group166}
S(\ell_1) = \{k_1, k_2\} \subseteq \eta(A_1) = A_2,
\end{equation*}
and therefore $A_1 = \{k_1\}$ is a proper subset of $\eta(A_1) = A_2$.  Hence we have $2 \le K$.

If $A_K$ is not minimal there exists a proper subset $B \subseteq A_K$ such that $\eta(B) = B$, and therefore $B$ belongs to $\pP$.  Let
\begin{equation}\label{group168}
C = A_K \setminus B = A_K \cap \bigl(\{1, 2, \dots , N\} \setminus B\bigr)
\end{equation}
be the complement of $B$ in $A_K$.  It follows from Lemma \ref{lemgroup1}, and the representation on the right of (\ref{group168}),
that $C$ is a proper subset of $A_K$ and $C$ belongs to $\pP$.  Thus we have the disjoint union of proper subsets
\begin{equation}\label{group170}
A_K = B \cup C,\quad\text{where $B \in \pP$ and $C \in \pP$}.
\end{equation}
Plainly $A_1 = \{k_1\}$ is a subset of either $B$ or $C$, and by renaming these sets if necessary we may assume that $A_1 = \{k_1\}$
is contained in $B$.  Then it follows from Lemma \ref{lemgroup2} that
\begin{equation*}\label{group175}
A_n \subseteq B\quad\text{for each $n = 1, 2, 3, \dots $}.
\end{equation*}
But this is inconsistent with the representation of $A_K$ as the disjoint union (\ref{group170}).  We conclude that $B$ and $C$ do not exist, 
and therefore $A_K$ is minimal.
\end{proof}

It follows from Lemma \ref{lemgroup3} that each element of $\{1, 2, \dots , N\}$ is contained in a minimal subset.  This
minimal subset is unique, and leads to a partition of $\{1, 2, \dots , N\}$ into a disjoint union of minimal subsets.

\begin{lemma}\label{lemgroup4}  Let $B$ and $C$ be nonempty, minimal subsets in $\pP$.  Then either
\begin{equation*}\label{group190}
B = C,\quad\text{or}\quad B \cap C = \emptyset.
\end{equation*}
\end{lemma}

\begin{proof}  If $B \cap C = \emptyset$ we are done.  Therefore we assume that $k_1$ is a point in $B \cap C$.    Let 
$A_1 = \{k_1\}$, and let $A_1, A_2, A_3, \dots $ be the sequence of subsets defined by (\ref{group160}).  Let $K$ be the smallest
positive integer such that (\ref{group162}) holds.  By Lemma \ref{lemgroup3} the subset $A_K$ is minimal, and by Lemma \ref{lemgroup2}
we have both $A_K \subseteq B$ and $A_K \subseteq C$.  But $A_K$ is minimal and therefore $A_K$ cannot be a proper subset of 
the minimal subset $B$.  Similarly, $A_K$ cannot be a proper subset of the minimal subset $C$.  We conclude that
\begin{equation*}\label{group200}
B = A_K = C.
\end{equation*}
This proves the lemma.
\end{proof}

\begin{lemma}\label{lemgroup5}  Let {\rm (\ref{group95})} be a collection of distinct subsets of $\{1, 2, \dots , N\}$ such that
\begin{equation*}\label{group230}
|S(\ell)| = 2\quad\text{for each $\ell = 1, 2, \dots , L$},
\end{equation*} 
and
\begin{equation*}\label{group232}
\bigcup_{\ell = 1}^L S(\ell)= \{1, 2, \dots , N\}.
\end{equation*}
Let $\pP \subseteq \A$ be the collection of subsets of $\{1, 2, \dots , N\}$ defined by {\rm (\ref{group142})}, and let $A_1, A_2, \dots , A_r$ 
be the collection of all distinct, minimal subsets in $\pP$.  Then the subsets $A_1, A_2, \dots , A_r$ are disjoint, and
\begin{equation*}\label{group235}
A_1 \cup A_2 \cup \cdots \cup A_r = \{1, 2, \dots , N\}.
\end{equation*}
\end{lemma}

\begin{proof}  The subsets $A_1, A_2, \dots , A_r$ exist by Lemma \ref{lemgroup3}.  Then it follows from Lemma \ref{lemgroup4} that the
subsets $A_1, A_2, \dots , A_r$ are disjoint.  Therefore we get 
\begin{equation}\label{group240}
A_1 \cup A_2 \cup \cdots \cup A_r \subseteq \{1, 2, \dots , N\}.
\end{equation}
It follows from Lemma \ref{lemgroup3} that each point in $\{1, 2, \dots , N\}$ is contained in a minimal subset, hence there is
equality in (\ref{group240}).
\end{proof}

We continue to assume that $L$ and $N$ are positive integers that satisfy (\ref{group107}).  Let $\bxi_1, \bxi_2, \dots , \bxi_L$ be vectors 
from the set of extreme points $F_N$, and write
\begin{equation*}\label{set1}
\Xi = \bigl(\bxi_1\ \bxi_2\ \cdots \ \bxi_L\bigr)
\end{equation*}
for the $N\times L$ matrix having $\bxi_1, \bxi_2, \dots , \bxi_L$ as columns.  We assume that no row of the matrix $\Xi$ is identically
zero, and we assume that $\rank \Xi = L$.  We write 
$\bxi_{\ell} = \bigl(\xi_{n \ell}\bigr)$ and use the vectors $\bxi_{\ell}$ to define a collection of subsets
\begin{equation}\label{set8}
S(\ell) \subseteq \{1, 2, \dots , N\},\quad\text{for each $\ell = 1, 2, \dots , L$}.
\end{equation}
More precisely, we define
\begin{equation}\label{set15}
S(\ell) = \{n : \text{$1 \le n \le N$ and $\xi_{n \ell} \not= 0$}\},\quad\text{for each $\ell = 1, 2, \dots , L$}.
\end{equation}
As each column vector $\bxi_{\ell}$ belongs to the set of extreme points $F_N$, it follows that each subset $S(\ell)$ has cardinality $2$, and
\begin{equation*}\label{set18}
\sum_{n = 1}^N \xi_{n \ell} = \sum_{n \in S(\ell)} \xi_{n \ell} = 0.
\end{equation*}
Because no row of the matrix $\Xi$ is identically zero, we find that
\begin{equation*}\label{set22}
\bigcup_{\ell = 1}^L S(\ell) = \{1, 2, \dots , N\}.
\end{equation*}
Therefore the subsets $S(\ell)$ defined by (\ref{set15}) satisfy the conditions (\ref{group100}) and (\ref{group105}) that were assumed in
the previous lemmas.  We continue to write $\A$ for the collection of all subsets of $\{1, 2, \dots , N\}$, and we write $\pP$ for the collection 
of subsets defined by (\ref{group142}).

Next we suppose that $A_1, A_2, \dots , A_r$ is the collection of distinct, nonempty, minimal subsets in $\pP$.  Then it follows from 
Lemma \ref{lemgroup5} that
\begin{equation}\label{set29}
A_1 \cup A_2 \cup \cdots \cup A_r = \{1, 2, \dots , N\}
\end{equation}
is a disjoint union of nonempty sets.  Because each subset $A_j$ is minimal we have
\begin{equation}\label{set36}
A_j = \bigcup_{\substack{\ell = 1\\S(\ell) \subseteq A_j}}^L S(\ell) = \bigcup_{\substack{\ell = 1\\S(\ell) \cap A_j \not= \emptyset}}^L S(\ell).
\end{equation} 
We use each subset $A_j$ to define a subset $D_j \subseteq \{1, 2, \dots , L\}$ by
\begin{equation}\label{set43}
D_j = \{\ell : \text{$1 \le \ell \le L$ and $S(\ell) \subseteq A_j$}\},\quad\text{for $j = 1, 2, \dots , r$}.
\end{equation}
Then it follows from (\ref{set29}), (\ref{set36}), and (\ref{set43}), that
\begin{equation}\label{set50}
D_1 \cup D_2 \cup \cdots \cup D_r = \{1, 2, \dots , L\}
\end{equation}
is a disjoint union of nonempty sets.  For each $j = 1, 2, \dots , r$ we write $Y_j$ for the $N \times |D_j|$ submatrix of $\Xi$ 
having columns indexed by the integers in $D_j$.  That is, we define
\begin{equation}\label{set71}
Y_j = \bigl(\bxi_{\ell}\bigr),\quad\text{where $\ell \in D_j$ indexes columns}.
\end{equation}
We assemble the matrices $Y_1, Y_2, \dots , Y_r$ as $N \times |D_j|$ blocks so as to define the $N \times L$ matrix
\begin{equation}\label{set78}
Z = \bigl(Y_1\ Y_2\ \cdots Y_r\bigr).
\end{equation}
Because of the disjoint union (\ref{set50}), the columns of the matrix $Z$ can also be obtained by permuting the columns of the 
matrix $\Xi$.  That is, there exists an $L \times L$ permutation matrix $P$ such that 
\begin{equation*}\label{set85}
\Xi = Z P.  
\end{equation*}
As $\det P = \pm 1$ and the columns of $\Xi$ are linearly independent, it follows that the matrix $Y_j$ has rank $|D_j|$
for each $j = 1, 2, \dots , r$.  We also find that
\begin{equation*}\label{set88}
\det \bigl(\Xi^T \Xi\bigr) = \det\bigl(P^T Z^T Z P\bigr) = \det\bigl(Z^T Z\bigr)
\end{equation*}
is a positive integer.

Now suppose that $1 \le i \le r$, that $1 \le j \le r$, and $i \not= j$.  It follows from (\ref{set8}), (\ref{set43}), and (\ref{set50}), 
that each nonzero row of the matrix $Y_i$ is indexed by an integer in the set $A_i$, and each nonzero row of the matrix $Y_j$ is indexed 
by an integer in the set $A_j$.  As $A_i$ and $A_j$ are disjoint we conclude that $Y_i^T Y_j$ is a zero matrix.  Because we have 
organized $Z$ into blocks as in (\ref{set78}), we find that
\begin{equation}\label{set92}
\det \bigl(\Xi^T \Xi\bigr) = \det \bigl(Z^T Z\bigr) = \prod_{j = 1}^r \det\bigl(Y_j^T Y_j\bigr).
\end{equation}

Since the extreme points $\bxi_l$ that form the columns of $\Xi$ belong to $F_N$, it follows that 
\begin{equation*}\label{set99}
\sum_{n = 1}^N \xi_{n \ell} = 0,\quad \text{for each $\ell = 1, 2, \dots L$}.
\end{equation*}
For each $j = 1, 2, \dots , r$ the nonzero rows of $Y_j$ are indexed by the elements of $A_j$, and so we get
\begin{equation}\label{set106}
\sum_{n \in A_j} \xi_{n \ell} = 0,\quad \text{for each $\ell \in D_j$}.
\end{equation}
As $Y_j$ has rank $|D_j|$ we find that
\begin{equation}\label{set113}
|D_j| + 1 \le |A_j|.
\end{equation}

Next we will show that there is equality in the inequality (\ref{set113}).  Each subset $A_j$ is minimal in $\pP$ and therefore no proper 
subset of $A_j$ belongs to $\pP$.  And it follows from (\ref{set106}) that the $|A_j|$ distinct (row) vectors
\begin{equation}\label{set127}
\big\{\bigl(\xi_{n \ell}\bigr) : \text{$n \in A_j$}\big\}
\end{equation}  
are linearly dependent.  Let $f : A_j \rightarrow \Z$ be a function that is supported on the subset
\begin{equation*}\label{set133}
B = \big\{n \in A_j : \text{$f(n) \not= 0$}\big\},
\end{equation*}
where $B$ is a proper subset of $A_j$.  As $B$ does not belong to $\pP$ it follows that there exists $\ell_1$ in $D_j$ such that
\begin{equation*}\label{set135}
\bigl|S(\ell_1) \cap B\bigr| = 1.
\end{equation*}
We conclude that
\begin{equation*}\label{set138}
\sum_{n \in A_j} f(n) \xi_{n \ell_1} = \sum_{n \in B} f(n) \xi_{n \ell_1} \not= 0,
\end{equation*}
because this sum contains exactly one nonzero term.  This shows that no proper subset of the collection of (row) vectors (\ref{set127})
is linearly dependent.  In particular, each subset of the (row) vectors in (\ref{set127}) with cardinality $|A_j| - 1$ is linearly independent.
As the rank of the matrix $Y_j$ is $|D_j|$ we conclude from (\ref{set113}) that 
\begin{equation}\label{set141}
|D_j| + 1 = |A_j|\quad\text{for each $j = 1, 2, \dots , r$}.
\end{equation}
We also get the identity
\begin{equation}\label{set147}
L + r = \sum_{j = 1}^r \bigl(|D_j| + 1\bigr) = \sum_{j = 1}^r |A_j| = N
\end{equation}
which determines the value of $r$.

\begin{lemma}\label{lemgroup6}  Let the columns of the $N \times L$ matrix
\begin{equation*}\label{set151}
\Xi = \bigl(\bxi_1\ \bxi_2\ \cdots \ \bxi_L\bigr)
\end{equation*}
be vectors from the set of extreme points $F_N$ defined in {\rm (\ref{reg40})}.  If $L < N \le 2L$ then
\begin{equation}\label{set158}
\big\|\bxi_1 \wedge \bxi_2 \wedge \cdots \wedge \bxi_L\big\|_1 \le \biggl(\frac{N}{N - L}\biggr)^{N - L}.
\end{equation}
\end{lemma}

\begin{proof}  Clearly we may assume that $\rank \Xi = L$.  And we assume to begin with that no row of the matrix $\Xi$ is identically zero.  
As in our proof of Lemma \ref{lemreg4} we have
\begin{equation}\label{set161}
\big\|\bxi_1 \wedge \bxi_2 \wedge \cdots \wedge \bxi_L\big\|_1 
		= \sum_{\substack{I \subseteq \{1, 2, \dots , N\}\\ |I| = L}} \bigl(\det \Xi_I\bigr)^2 = \det\bigl(\Xi^T \Xi\bigr)
\end{equation}
by the Cauchy-Binet identity.  By combining (\ref{set92}) and (\ref{set161}) we find that
\begin{equation*}\label{set165}
\big\|\bxi_1 \wedge \bxi_2 \wedge \cdots \wedge \bxi_L\big\|_1 = \prod_{j = 1}^r \det\bigl(Y_j^T Y_j\bigr),
\end{equation*}
where each $N \times |D_j|$ matrix $Y_j$ is defined as in (\ref{set71}).  Let $W_j$ be the $|A_j| \times |D_j|$ submatrix of $Y_j$ obtained
by removing all rows which are identically zero.  Because there is equality in the inequality (\ref{set113}) the submatrix $W_j$ is also 
$(|D_j| + 1) \times |D_j|$.  That is, $W_j$ is an $(M + 1) \times M$ matrix with columns in the set of extreme points $F_M$, where $M = |D_j|$.  
Then it follows from the inequality (\ref{turn31}) and (\ref{set141}) that
\begin{equation}\label{set162}
\begin{split}
\prod_{j = 1}^r \det\bigl(Y_j^T Y_j\bigr) &= \prod_{j = 1}^r \det\bigl(W_j^T W_j\bigr)\\
							   &\le \prod_{j = 1}^r \bigr(|D_j| + 1\bigr)\\
							   &= \prod_{j = 1}^r |A_j|.
\end{split}
\end{equation}
We estimate the product on the right of (\ref{set162}) by applying the arithmetic/geometric mean inequality and using the identity
(\ref{set147}).  In this way we arrive at the inequality							   
\begin{equation*}\label{set169}
\begin{split}
\prod_{j = 1}^r \det\bigl(Y_j^T Y_j\bigr) &\le \biggl(r^{-1} \sum_{j = 1}^r |A_j|\biggr)^r\\
							   &= \bigl(r^{-1} N\bigr)^r\\
							   &= \biggl(\frac{N}{N - L}\biggr)^{N - L}.
\end{split}	
\end{equation*}  
This proves (\ref{set158}) under the assumption that no row of $\Xi$ is identically zero.

Next we suppose that $L  < N \le 2L$, that
\begin{equation*}\label{ring234}
\Xi = \bigl(\bxi_1\ \bxi_2\ \cdots \ \bxi_L\bigr)
\end{equation*}
is an $N \times L$ matrix with columns $\bxi_1, \bxi_2, \dots , \bxi_L$ from $F_N$, that $\rank \Xi = L$, and that $\Xi$ has exactly 
$N - M > 0$ rows that are identically zero.  Because $\rank \Xi = L$, we find that $L \le M < N \le 2L$.  We write
\begin{equation*}\label{ring239}
\Xi ^{\prime}= \bigl(\bxi_1^{\prime}\ \bxi_2^{\prime}\ \cdots \ \bxi_L^{\prime}\bigr)
\end{equation*}
for the $M \times L$ matrix obtained from $\Xi$ by removing the rows of $\Xi$ that are identically zero.  It follows
from Lemma \ref{lemreg2} that each column $\bxi_1^{\prime}, \bxi_2^{\prime}, \dots , \bxi_L^{\prime}$ belongs to $F_M$.
Clearly each $L \times L$ submatrix of $\Xi$ with a row that is identically zero has a zero determinant.  Thus we have
\begin{equation*}\label{ring242}
\|\bxi_1\wedge \bxi_2 \wedge \cdots \wedge \bxi_L\|_1 = \|\bxi_1^{\prime} \wedge \bxi_2^{\prime} \wedge \cdots \wedge \bxi_L^{\prime}\|_1.
\end{equation*}
If $L = M$ then $\Xi ^{\prime}$ is $L \times L$, and it follows from Lemma \ref{lemreg3} that
\begin{equation*}\label{ring244}
\|\bxi_1^{\prime} \wedge \bxi_2^{\prime} \wedge \cdots \wedge \bxi_L^{\prime}\|_1 = 1 \le \biggl(\frac{N}{N - L}\biggr)^{N-L}.
\end{equation*}
If $L < M < N \le 2L$ then by the case already considered we get
\begin{equation*}\label{ring249}
\begin{split}
\|\bxi_1^{\prime} \wedge \bxi_2^{\prime} \wedge \cdots \wedge \bxi_L^{\prime}\|_1
	&\le \biggl(\frac{M}{M - L}\biggr)^{M-L}\\
	&< \biggl(\frac{N}{N - L}\biggr)^{N-L}.
\end{split}
\end{equation*}
This verifies the bound (\ref{set158}) in general.
\end{proof}

We now combine Lemma \ref{lemreg5} and Lemma \ref{lemgroup6} to obtain the inequality (\ref{set158}) in full generality.

\begin{theorem}\label{thmset1}  Let the columns of the $N \times L$ matrix
\begin{equation*}\label{ring253}
\Xi = \bigl(\bxi_1\ \bxi_2\ \cdots \ \bxi_L\bigr)
\end{equation*}
be vectors in the set of extreme points $E_N \cup F_N$.  If $L < N \le 2L$ then
\begin{equation}\label{ring255}
\|\bxi_1\wedge \bxi_2 \wedge \cdots \wedge \bxi_L\|_1 \le \biggl(\frac{N}{N-L}\biggr)^{N - L}.
\end{equation}
\end{theorem}

\begin{proof}  We argue by induction on the positive integer $L$.  If $L = 1$ then $N = 2$ and the result is trivial to check.  
Next we assume that $2 \le L$, and we assume that (\ref{ring255}) holds for all pairs $(L^{\prime}, N^{\prime})$ such that 
$L^{\prime} < N^{\prime} \le 2L^{\prime}$ and $1 \le L^{\prime} < L$. 

If the extreme points $\bxi_1, \bxi_2, \dots , \bxi_L$ all belong 
to the set of extreme points $E_N$, then
\begin{equation*}\label{ring303}
\|\bxi_1 \wedge \bxi_2 \wedge \cdots \wedge \bxi_L\|_1 = 1
\end{equation*}
and the inequality (\ref{ring255}) is trivial.  If the extreme points $\bxi_1, \bxi_2, \dots , \bxi_L$ all belong to the set of extreme points $F_N$, 
then the inequality (\ref{ring255}) follows from Lemma \ref{lemgroup6}.  To complete the proof we assume that $K$ of the extreme 
points $\bxi_1, \bxi_2, \dots , \bxi_L$ belong to $E_N$, and $L - K$ extreme points $\bxi_1, \bxi_2, \dots , \bxi_L$ belong to $F_N$, 
where $1 \le K < L$.  In this case the set of extreme points satisfies the hypotheses of Lemma \ref{lemreg5}.  It follows from the conclusion of 
Lemma \ref{lemreg5} that there exist linearly independent extreme points $\bbeta_1, \bbeta_2, \dots , \bbeta_{L-K}$ in the set 
$E_{N-K} \cup F_{N-K}$ such that
\begin{equation}\label{ring308}
\|\bxi_1 \wedge \bxi_2 \wedge \cdots \wedge \bxi_L \|_1 = \|\bbeta_1 \wedge \bbeta_2 \wedge \cdots \wedge \bbeta_{L-K}\|_1.
\end{equation}

We write $L^{\prime} = L - K$, $N^{\prime} = N - K$, and we consider two cases.  First we suppose that 
\begin{equation*}
N^{\prime} \le 2L^{\prime}.
\end{equation*}
In this case we apply the inductive hypothesis and conclude that
\begin{equation}\label{ring313}
\begin{split}
 \|\bbeta_1 \wedge \bbeta_2 \wedge \cdots \wedge \bbeta_{L-K}\|_1 
 			       &\le \biggl(\frac{N^{\prime}}{N^{\prime} - L^{\prime}}\biggr)^{N^{\prime} - L^{\prime}}\\
                                &= \biggl(\frac{N - K}{N - L}\biggr)^{N - L}\\
                                &< \biggl(\frac{N}{N - L}\biggr)^{N - L}.
\end{split}
\end{equation}
Next we suppose that
\begin{equation*}\label{ring315}
2L^{\prime} \le N^{\prime}.
\end{equation*}
In this case we appeal to the inequality (\ref{reg271}) which we have already proved.  By that result we have
\begin{equation}\label{ring318}
\begin{split}
 \|\bbeta_1 \wedge \bbeta_2 \wedge \cdots \wedge \bbeta_{L-K}\|_1 &\le 2^{L^{\prime}}\\
 				&= \min\bigg\{2^{L^{\prime}}, \biggl(\frac{N^{\prime}}{N^{\prime}-L^{\prime}}\biggr)^{N^{\prime} - L^{\prime}}\bigg\}\\
                                  &\le \biggl(\frac{N^{\prime}}{N^{\prime}-L^{\prime}}\biggr)^{N^{\prime} - L^{\prime}}\\
                                  &= \biggl(\frac{N - K}{N - L}\biggr)^{N - L}\\
                                  &< \biggl(\frac{N}{N - L}\biggr)^{N - L}.
\end{split}
\end{equation}
Combining (\ref{ring308}), (\ref{ring313}), and (\ref{ring318}), establishes the inequality
\begin{equation*}\label{ring323}
\|\bxi_1 \wedge \bxi_2 \wedge \cdots \wedge \bxi_L \|_1 \le \biggl(\frac{N}{N - L}\biggr)^{N - L}
\end{equation*}
whenever $L < N \le 2L$.  This proves the lemma.
\end{proof}

If $\bx_1, \bx_2, \dots , \bx_L$ belong to $\R^N$ and $L < N \le 2L $, then it follows from (\ref{ring255}) that
\begin{equation*}\label{ring329}
\|\bx_1 \wedge \bx_2 \wedge \cdots \wedge \bx_L\|_1 \le \biggl(\frac{N}{N - L}\biggr)^{N - L}\ \delta(\bx_1) \delta(\bx_2) \cdots \delta(\bx_L).
\end{equation*}
This proves the inequality (\ref{norm38}), and so completes the proof of Theorem \ref{thmnorm1}.  

\section{Proof of Theorem \ref{thmfirst1}}

We apply Theorem \ref{thmnorm1} with $N = r + 1$ and $L = q$, and we apply the theorem to the collection of linearly independent
points $\balpha_1, \balpha_2, \dots , \balpha_q$ in 
\begin{equation*}\label{wedge1}
\Gamma_S(k) \subseteq \D_r \subseteq \R^{r+1}.
\end{equation*}
From (\ref{norm44}) we find that
\begin{equation}\label{wedge6}
\begin{split}
\|\balpha_1 \wedge \balpha_2 &\wedge \cdots \wedge \balpha_q\|_1\\
	&\le \min\bigg\{2^q, \biggl(\frac{r+1}{r+1-q}\biggr)^{r+1-q}\bigg\}\ \delta(\balpha_1) \delta(\balpha_2) \cdots \delta(\balpha_q)\\
	&= C(r, q)\  \delta(\balpha_1)\delta(\balpha_2) \cdots \delta(\balpha_q).
\end{split}
\end{equation}
By the product formula the points $\balpha_1, \balpha_2, \dots , \balpha_q$ belong to the diagonal subspace $\D_r$.  Therefore we get
\begin{equation}\label{wedge11}
\delta(\balpha_j) = \hh \|\balpha_j\|_1,\quad\text{for each $j = 1, 2, \dots , q$}.
\end{equation}
Combining (\ref{wedge6}) and (\ref{wedge11}) establishes the inequality (\ref{first89}).

\section{Proof of Theorem \ref{thmfirst2}}
 
Let $1 \le L < N$ and let
\begin{equation*}\label{McM5}
X = \bigl(\bx_1\ \bx_2\ \cdots \ \bx_L\bigr)
\end{equation*}
be an $N \times L$ real matrix with columns $\bx_1, \bx_2, \dots , \bx_L$.  We assume that the columns of $X$ are
$\R$-linearly independent so that $\rank X = L$, and
\begin{equation*}\label{McM10}
\bx_1 \wedge \bx_2 \wedge \cdots \wedge \bx_L \not= \bo.
\end{equation*}
We use the matrix $X$ to define a norm on $\R^L$ by
\begin{equation}\label{McM16}
\bwy \mapsto \|X \bwy\|_1.
\end{equation}
The unit ball associated to the norm (\ref{McM16}) is obviously the set
\begin{equation*}\label{McM21}
B_X = \big\{\bwy \in \R^L : \|X \bwy\|_1 \le 1\big\}.
\end{equation*}
And it is not difficult to show that the dual unit ball is
\begin{equation*}\label{McM26}
B_X^* = \big\{X^T \bw : \text{$\bw \in \R^N$ and $\|\bw\|_{\infty} \le 1$}\big\}.
\end{equation*}
It can be shown (see \cite{bolker1969}, \cite{schneider1983}, or for a more general result \cite[Lemma 2]{vaaler2014}) 
that the dual unit ball $B_X^*$ is an example of a 
zonoid.  Therefore by an inequality of S. Reisner \cite[Theorem 2]{reisner1985}, we have
\begin{equation}\label{McM31}
\frac{4^L}{L!} \le \Vol_L\bigl(B_X\bigr) \Vol_L\bigl(B_X^*\bigr).
\end{equation}
An identity for the $L$-dimensional volume of $B_X^*$ was established by P.~McMullen and C.~G.~Shephard as discussed in
\cite{McMullen1984} and \cite[equation (57)]{shephard1974}.
Their result asserts that
\begin{equation}\label{McM37}
\Vol_L\bigl(B_X^*\bigr) =  2^L \sum_{|I| = L} \bigl|\det X_I\bigr| = 2^L \|\bx_1 \wedge \bx_2 \wedge \dots \wedge \bx_L\|_1.
\end{equation}
By combining Reisner's inequality (\ref{McM31}) and the volume formula (\ref{McM37}), we obtain the lower bound
\begin{equation}\label{McM39}
\frac{2^L}{L!} \le \Vol_L\bigl(B_X\bigr) \|\bx_1 \wedge \bx_2 \wedge \cdots \wedge \bx_L\|_1.
\end{equation}

Now let
\begin{equation*}\label{McM41}
0 < \lambda_1 \le \lambda_2 \le \cdots \le \lambda_L < \infty
\end{equation*}
be the successive minima for the convex symmetric set $B_X$ and the integer lattice $\Z^L$.  By Minkowski's theorem 
on successive minima (see \cite[Section VIII.4.3]{cassels1971}) we have
\begin{equation}\label{McM44}
\Vol_L\bigl(B_X\bigr) \lambda_1 \lambda_2 \cdots \lambda_L \le 2^L.
\end{equation}
We combine the lower bound (\ref{McM39}) and the upper bound (\ref{McM44}), and obtain the inequality
\begin{equation}\label{McM47}
\lambda_1 \lambda_2 \cdots \lambda_L \le \bigl(L!\bigr) \|\bx_1 \wedge \bx_2 \wedge \cdots \wedge \bx_L\|_1.
\end{equation}
This leads to the following general result.

\begin{theorem}\label{thmMcM1}  Let $\X \subseteq \R^N$ be the free group of rank $L$ generated by the linearly independent
vectors $\bx_1, \bx_2, \dots , \bx_L$.  Then there exist linearly independent points $\bwy_1, \bwy_2, \dots , \bwy_L$ in $\X$ such that
\begin {equation}\label{McM54}
\|\bwy_1\|_1 \|\bwy_2\|_1 \cdots \|\bwy_L\|_1 \le (L!) \|\bx_1 \wedge \bx_2 \wedge \cdots \wedge \bx_L\|_1.
\end{equation}
Moreover, if $\Y \subseteq \X$ is the subgroup generated by the points $\bwy_1, \bwy_2, \dots , \bwy_L$, then 
$[\X : \Y] \le L!$.
\end{theorem}

\begin{proof}  By Minkowski's theorem on successive minima there exist linearly independent points $\bm_1, \bm_2, \dots , \bm_L$ in 
the integer lattice $\Z^L$ such that
\begin{equation}\label{McM49}
\big\|X \bm_{\ell}\|_1 = \lambda_{\ell}\quad\text{for $\ell = 1, 2, \dots , L$}.
\end{equation}
As $\rank X = L$ the points
\begin{equation*}\label{McM53}
\big\{X \bm_{\ell} : \ell = 1, 2, \dots , L\big\}
\end{equation*}
are linearly independent points in the free abelian group $\X$.  We write $\bwy_{\ell} = X\bm_{\ell}$ for each $\ell = 1, 2, \dots , L$.
Then (\ref{McM54}) follows from (\ref{McM47}) and (\ref{McM49}).  The bound $[\X : \Y] \le L!$ also follows from 
Minkowski's theorem.
\end{proof}

Now let $L = q$, $N = r + 1$ and let $\fA \subseteq \R^{r+1}$ be the subgroup of rank $q$ generated by the linearly independent vectors
$\balpha_1, \balpha_2, \dots , \balpha_q$. 
By Theorem \ref{thmMcM1} there exist linearly independent vectors $\Bbeta_1, \Bbeta_2, \dots , \Bbeta_q$ in $\fA$ such that
\begin{equation*}\label{app60}
\|\Bbeta_1\|_1 \|\Bbeta_2\|_1 \cdots \|\Bbeta_q\|_1 \le (q!) \|\balpha_1 \wedge \balpha_2 \wedge \cdots \wedge \balpha_q\|_1.
\end{equation*}                       
Moreover, the free group $\fB \subseteq \fA$ generated by the vectors $\Bbeta_1, \Bbeta_2, \dots , \Bbeta_q$, has rank $q$ and index
\begin{equation*}\label{app65}
[\fA : \fB] \le q!.
\end{equation*}
This proves Theorem \ref{thmfirst2}.


\end{document}